\documentclass[l
eqno,11pt]{amsart}

\usepackage{leftidx}
\usepackage{comment}
\usepackage{amsmath}
\usepackage{graphicx, color}
\usepackage{amscd}
\usepackage{amsxtra}
\usepackage{amsfonts}
\usepackage{amssymb}
\usepackage{mathrsfs}
\usepackage{mathtools}
\usepackage{ulem}
\usepackage{enumitem}
\usepackage[colorlinks=true,linkcolor=blue]{hyperref}
\usepackage{todonotes}
\usepackage{mathtools}

\newtheorem{thm}{Theorem}[section]

\newtheorem{lem}[thm]{Lemma}

\newtheorem{defn}[thm]{Definition}
\newtheorem{rem}[thm]{Remark}
\newtheorem{exam}[thm]{Example}
\numberwithin{equation}{section}

\newcommand\be{\begin{equation}}
\newcommand\ee{\end{equation}}
\newcommand\R{\mathbb R}
\newcommand\N{\mathbb N}
\newcommand\rn{\mathbb R^n}
\newcommand\rnn{\mathbb R^{n+1}}
\newcommand\rnp{\mathbb R^{n+1}_+}
\newcommand\X{\mathbf {X}}
\newcommand\Y{\mathbf {Y}}
\newcommand\Z{\mathbf {Z}}
\newcommand\V{\mathbf {V}}

\newcommand{\n}{\nabla_2}
\newcommand{\tr}{\mathbf{Tr}\hspace{0.03cm} u}

\allowdisplaybreaks

\title[]
{Strongly nonlinear Robin problems for  harmonic and polyharmonic functions in the half-space
}

\author{Andrea Cianchi}
\address{Dipartimento di Matematica e Informatica \lq \lq U. Dini''\\
Universit\`a di Firenze\\  Viale Morgagni 67/A, 50134 Firenze,
Italy}
 \email{andrea.cianchi@unifi.it}

\author{Gael Y.  Diebou}
\address{Department of Mathematics\\University of Toronto\\ 40 St George Street, Toronto, ON, M5S 2E4, Canada.}
\email{gaely.diebou@utoronto.ca}

\author{Lenka Slav\'ikov\'a}

\address{Department of Mathematical Analysis
\\  Faculty of Mathematics and Physics,
Charles University
\\  Sokolovsk\'a~83,
186~75 Praha~8,
Czech Republic}
\email{slavikova@karlin.mff.cuni.cz}


\begin{document}

\begin{abstract} Existence and global regularity results for  boundary-value  problems of Robin type for harmonic and polyharmonic functions in $n$-dimensional half-spaces are offered. 
The Robin condition on the normal derivative on the boundary of the half-space is prescribed by a nonlinear function $\mathcal N$ of the relevant harmonic or polyharmonic functions. General Orlicz type growths for the function $\mathcal N$ are considered. For instance, functions $\mathcal N$ of classical power type, their perturbations by logarithmic factors, and exponential functions are allowed. New sharp boundedness properties in Orlicz spaces of some classical operators from harmonic analysis, of independent interest, are critical for our approach.
\end{abstract}

\maketitle


\section{Introduction}\label{intro}

We deal with the existence and global regularity of harmonic and polyharmonic functions in $n$-dimensional half-spaces, subject to general nonlinear Robin type boundary conditions. 

 The problems under consideration for harmonic functions have the form
\begin{align}\label{eq:main-eq}
		\begin{cases}
			\Delta u=0 & \quad \mbox{ in }\,\,\rnp\\
			-\dfrac{\partial u}{\partial x_{n+1}}
   = \mathcal N(u)+f & \quad \mbox{ on }\,\,\partial\rnp.
		\end{cases}
	\end{align}
 Here,    $\mathcal N: \mathbb R \to \mathbb R$ is a locally Lipschitz continuous function,  $f:\rn \to \mathbb R$ is a locally integrable function, $\rnp$ denotes the half-space  of those points in $\mathbb R^{n+1}$ whose last component is positive, and $n\geq 2$.

If $\mathcal N=0$, then \eqref{eq:main-eq} is a classical Neumann boundary value problem treated in \cite{Arm}. The case when $f= 0$ and $\mathcal N$ is a power-type nonlinearity of the form 
\begin{equation}\label{power}
\mathcal N(t)= |t|^{p-1}t  
\end{equation}
has also been investigated for subcritical ($p< (n+1)/n$) and critical  ($p=(n+1)/n$) exponents, see  \cite{Hu,QR} and the references  therein. 
In general, even for $\mathcal N$ as in \eqref{power}, the solvability of problem \eqref{eq:main-eq} is very much dependent on the inhomogeneous term $f$. Under suitable assumptions on the latter,  results on the existence, uniqueness, regularity, and qualitative properties of solutions are available in the literature \cite{DL, FMM, Di1}.

The punctum of the present contribution is that the function $\mathcal N$ need not have a polynomial growth. For instance, the exponential type behavior 
\begin{equation}\label{exp}
\mathcal N(t) \approx e^{t^\alpha} \quad \text{near infinity}
\end{equation}
for some $\alpha >0$ is included in our discussion, and was the original motivation for this work. Elliptic equations with exponential Robin boundary conditions arise in diverse areas. In Riemannian geometry, the problem of finding conformal metrics in the two dimensional half-space (and, more generally, on manifolds with boundary) with constant Gauss curvature and boundary geodesic curvature is related to the solvability of problems of the type \eqref{eq:main-eq}, 
 with $\mathcal N(t)=\lambda e^t$ for some $\lambda\in \R$ \cite{Ch,N,Zh}. Certain models in mathematical physics also call into play exponential Robin boundary conditions. They describe the corrosion and oxidation phenomena of materials in electrochemistry \cite{MV,VX}. 

The problem \eqref{eq:main-eq}, with
$$\mathcal N(t)=\lambda\, te^{t^2},$$
for some $\lambda>0$, is prototypical and can help  grasp the main features and difficulties of the general setting under consideration. It is the Euler-Lagrange equation of the energy functional 
\begin{align}\label{trace}
E_{\lambda}(u)=\dfrac{\lambda}{2}\int_{\rnp}|\nabla u|^2dxdx_{n+1}-\dfrac{\lambda}{2}\int_{\rn}e^{u^2}  -1\,dx.
\end{align}
This functional is well defined for $n=1$, thanks to a borderline trace embedding, which ensures that the integral over $\mathbb R^2$ on the right-hand side of \eqref{trace} converges for every $u\in W^{1,2}(\mathbb R^2)$.
As a consequence, variational methods apply. 
However, this is not guaranteed in higher dimensions, and classical direct methods of the calculus of variations fail. This drawback also surfaces when dealing with functions $\mathcal N$ with more general non-standard growths.

The approach to the problem \eqref{eq:main-eq}  adopted in this paper has a non-variational nature and is based on a combination of  potential theoretic techniques,  sharp non-standard  results from the theory of Sobolev type spaces, and fixed-point arguments. Novel developments on these methods in the framework of Orlicz spaces are a central step in our analysis.   


Our existence result on the problem \eqref{eq:main-eq} ensures that it admits a solution $u$, provided that the datum $f$ has a sufficiently small norm in a suitable space modeled on the function $\mathcal N$. 
The unconventional Sobolev type  space for the solution also depends on $\mathcal N$ and is equipped with a  norm which depends on both the gradient of functions in $\rnp$ and their trace  on $\partial\rnp$.
Due to the generality of the  admissible nonlinearities $\mathcal N$, the ambient spaces for $f$ and $u$ are naturally defined in terms of Orlicz norms. 

Stronger integrability properties of $f$ are shown to be reflected in a higher degree of integrability of the gradient of $u$ and of the trace of $u$ on $\partial\rnp$. Their integrability is again described in terms of finiteness of Orlicz norms. This is the content of our second result about the problem \eqref{eq:main-eq}.

Finally,  we prove that weak differentiability of the datum $f$ and its membership in an Orlicz-Sobolev space on $\rn$ endowed with a sufficiently strong norm guarantee that  $u$ is, in fact, a classical -- namely in $C^{\infty}(\rnp)\cap C^1(\overline{\rnp})$ -- solution to the problem \eqref{eq:main-eq}. The latter result also requires that $\mathcal N$ be differentiable, and the Orlicz-Sobolev space for $f$ is also related to the growth of the derivative $\mathcal N'$.

To give the flavour of the conclusions that can be deduced from these results, consider nonlinearities $\mathcal N$ with an exponential growth near infinity as in \eqref{exp}. Since the ambient space $\rnp$ has infinite measure, also the behaviour of $\mathcal N$ near zero is relevant. Assume, for instance, that $\mathcal N$ has a polynomial decay at zero, namely
\begin{align}
    \label{2024-404}
    \mathcal N(t)  =   |t|^{p-1} t\, e^{t^\alpha},
\end{align}
for some $p>n'$ and $\alpha >0$, where $n'=\frac n{n-1}$. As shown in Example \ref{ex3}, the problem \eqref{eq:main-eq} admits a weak solution $u$, provided that  
$$f \in L^{\frac n{p'}}(\rn)\cap L^n(\log L)^{n-1-\frac n\alpha}(\rn)$$ and has a sufficiently small norm. 
Here, $L^n(\log L)^{n-1-\frac n\alpha}(\rn)$ denotes a Zygmund space. Under the stronger Sobolev regularity assumption  that   $$
    f\in W^{1,\frac n{p'}}(\rn)\cap W^1L^n(\log L)^{n-1+\beta}(\rn),
$$  for some $\beta>0$, with a norm small enough, one has that  $u \in C^{\infty}(\rnp)\cap C^1(\overline{\rnp})$, an hence it is also a classical solution.  Observe that the intersection spaces for $f$ agree, up to equivalent norms, with suitable Orlicz spaces and Orlicz-Sobolev spaces.

The existence and regularity results about the problem \eqref{eq:main-eq}  outlined above admit analogs for parallel problems for polyharmonic equations. These problems  read:
\begin{align}\label{eq:gen-eq}
\begin{cases}
\Delta^{m} u=0& \quad \mbox{ in }\,\,\mathbb{R}^{n+1}_{+}\\
\dfrac{\partial  \Delta^{k}u}{\partial x_{n+1}}=0, \,\ k=0, \dots , m-2, & \quad \mbox{ on }\,\,\partial\mathbb{R}^{n+1}_{+}
\\ 
(-1)^{m}\dfrac{\partial \Delta^{m-1}u}{\partial x_{n+1}}=\mathcal N(u)+f   & \quad  \mbox{ on }\,\,\partial\mathbb{R}^{n+1}_{+}
\end{cases}
\end{align}
for $2m<n+1$, where $\Delta^{m}$ denotes the $m$-th order Laplace operator, obtained on iterating the Laplacian $m$ times. For power type nonlinearities $\mathcal N$ as in \eqref{power}, this problem can be regarded as the inhomogeneous counterpart of the higher-order boundary conformally invariant $Q$-curvature problem \cite{SX}.  In the range $\frac{n}{n-2m+1}<p<\infty$, and
for broad classes of boundary data $f$,
the solvability of problem \eqref{eq:gen-eq} follows as a special case of results of  \cite{Di}.  An analysis of the problem \eqref{eq:gen-eq} for non-polynomial nonlinearities $\mathcal N$ seems to be missing in the literature and is addressed in this paper.

The paper is structured as follows. The next section is devoted to notations and background from the theory of Young functions and Orlicz and 
Orlicz-Sobolev spaces. Our main results are stated in Section \ref{S:3}. The subsequent Section \ref{S:aux}
contains technical results about Sobolev conjugates of Young functions satisfying specific assumptions.  Boundedness properties of some classical operators from harmonic analysis, in Orlicz spaces, are established in Section \ref{S:4}. They are central in view of our approach, which makes substantial use of representation formulas via integral operators.
The proofs of the main results about  the problems \eqref{eq:main-eq}
and \eqref{eq:gen-eq}
are collected in Sections \ref{second-order} and \ref{higher-order}, respectively. Although problem \eqref{eq:main-eq} is a special case of \eqref{eq:gen-eq} (the condition on the second line being absent in this case), we prefer to offer a self-contained treatment of the former. It involves simpler notation, and 
readers who are just interested in second-order equations will find a proof that does not need a detour through the higher-order case. 

\section{Functional background}  

\subsection{Young functions}\label{SS:young}

The class of Orlicz spaces extends that of Lebesgue spaces in that the role of powers in the definition of the norm is played by more general Young functions. 
A
function $A: [0, \infty ) \to [0, \infty ]$ is called a Young
function if it is convex (non trivial), left-continuous and
vanishes at $0$. 
Any Young function takes the form
\begin{equation}\label{young}
A(t) = \int _0^t a(\tau ) \,\mathrm{d}\tau \qquad \hbox{for $t \geq
0$},
\end{equation}
for some non-decreasing, left-continuous function $a: [0, \infty )
\to [0, \infty ]$ which is neither identically equal to $0$ nor to
$\infty$. 
\\ The function
\begin{align}
    \label{2024-5}
    \frac{A(t)}{t} \quad \text{is non-decreasing.}
\end{align}
One has that
\begin{equation}\label{aA}
A(t) \leq  ta(t) \leq A(2t)\qquad \hbox{for $t
> 0$.}
\end{equation}
Moreover,
\begin{equation}\label{Ak}
\lambda A(t) \leq A(\lambda t) \quad \hbox{for $\lambda \geq 1$ and $t \geq0$.}
\end{equation}
The Young conjugate $\widetilde{A}$ of $A$  is defined by
$$\widetilde{A}(t) = \sup \{\tau t-A(\tau ):\,\tau \geq 0\} \qquad {\rm for}\qquad  t\geq 0\,.$$
Note the representation formula
\begin{equation}\label{youngconj}
\widetilde A(t) = \int _0^t a^{-1}(\tau ) d\tau \qquad \quad
\hbox{for $t \geq 0$},
\end{equation}
 where $a^{-1}$ denotes the (generalized) left-continuous inverse of the
function $a$ appearing in \eqref{young}. 
One can show that
\begin{equation}\label{AAtilde}
t \leq A^{-1}(t) \widetilde A^{-1}(t) \leq 2t \qquad \hbox{for $ t
\geq 0$,}
\end{equation}
where $A^{-1}$ and $\widetilde A^{-1}$ stand for the generalized
right-continuous inverses of $A$ and $\widetilde A$, respectively.
		\\
A Young function $A$ is said to satisfy the $\Delta_2$-condition --  briefly $A \in \Delta_2$ -- globally if there exists a constant $c$ such that
\begin{equation}\label{delta2}
A(2t) \leq c A(t) 
\end{equation}
 for $t \geq 0$.
 \\
The function $A$ is said to satisfy the $\nabla_2$-condition --  briefly $A \in \nabla_2$ -- globally if there exists a constant $c>2$ such that
\begin{equation}\label{nabla2}
A(2t) \geq c A(t) 
\end{equation}
for $t \geq 0$.
\\ One has that 
\begin{equation}\label{equivdelta2}
\text{$A \in \Delta_2$ if and only if there exists $p\geq 1$ such that $\sup_{t>0}\frac{ta(t)}{A(t)} \leq p$,}
\end{equation}
and
\begin{equation}\label{equivnabla2}
\text{$A \in \nabla_2$ if and only if there exists $p> 1$ such that $\inf_{t>0}\frac{ta(t)}{A(t)} \geq p$.}
\end{equation}
The $\Delta_2$-condition and the $\nabla_2$-condition  are said to be satisfied near infinity or near $0$ if there exists $t_0>0$ such that \eqref{delta2} or \eqref{nabla2} hold for $t\in (t_0, \infty)$ or for $t\in [0, t_0)$, respectively. Characterizations analogous to 
 \eqref{equivdelta2} and \eqref{equivnabla2} hold, with $\sup_{t>0}$ and $\inf_{t>0}$ replaced by $\sup$ and $\inf$ over the corresponding interval of values of $t$.

\begin{lem}\label{L:lemma_young}
Let $A$ be a Young function having the form \eqref{young} and let $K>1$.  
Then:
\\ (i) $A\in \Delta_2$ if and only if there exists a constant $c_1$ such that
\begin{equation}\label{E:delta2}
a(t) \geq \frac{1}{c_1}a(Kt) \quad \text{for} \quad t\geq 0.
\end{equation}
\\ (ii)  $A\in \nabla_2$ if and only if there exists a constant $c_2$ such that
\begin{equation}\label{E:nabla2}
a(t) \leq  \frac{1}{K}a(c_2t) \quad \text{for} \quad t\geq 0. 
\end{equation}
\end{lem}

\begin{proof}
Part (i) is well-known, and easily verified, thanks to equations \eqref{young} and \eqref{aA}.
As for Part (ii), one has that $A \in \nabla_2$ if and only if $\widetilde A \in \Delta_2$. Thus, by the formula \eqref{youngconj} and Part (i),
 $A \in \nabla_2$ if and only if $a^{-1}(Kt) \leq c_2 a^{-1}(t)$ for some constant $c_2>0$. 
The latter condition is
 in turn equivalent to~\eqref{E:nabla2}.
\end{proof}

\par\noindent
A  Young function $A$ is said to dominate another Young function $B$
globally   if there exists a positive constant $c$  such that
\begin{equation}\label{B.5bis}
B(t)\leq A(c t) 
\end{equation}
for $t \geq 0$. The function $A$ is said to dominate $B$ near infinity if there
exists $t_0\geq 0$ such that \eqref{B.5bis} holds for $t \geq t_0$. If $A$ and $B$ dominate each other globally [near infinity], then they are called equivalent globally [near infinity]. Equivalence in this sense will be denoted by
$$A \simeq B.$$
This terminology and notation will also be adopted for merely nonnegative functions, which are not necessarily Young functions.
\\
The global upper and lower
Matuszewska-Orlicz indices $I(A)$ and $i(A)$ of a (non-necessarily  Young) function $A$, which is  strictly positive and finite-valued in $(0,\infty)$, can be defined as
\begin{align}\label{index}
I(A)=\lim_{\lambda \to \infty} \frac{\log \left(\displaystyle\sup_{t>0} \frac{A(\lambda t)}{A(t)}\right)}{\log \lambda} \quad \text{and}
\quad
i(A)=\lim_{\lambda \to 0^+} \frac{\log \left(\displaystyle\sup_{t>0} \frac{A(\lambda t)}{A(t)}\right)}{\log \lambda}.
\end{align}
The indices $I_0(A)$ and $i_0(A)$ near zero are given by
\begin{align}\label{index0}
I_0(A)=\lim_{\lambda \to \infty} \frac{\log \left(\displaystyle\limsup_{t\to 0^+} \frac{A(\lambda t)}{A(t)}\right)}{\log \lambda} \quad \text{and}
\quad
i_0(A)=\lim_{\lambda \to 0^+} \frac{\log \left(\displaystyle\limsup_{t\to 0^+} \frac{A(\lambda t)}{A(t)}\right)}{\log \lambda}.
\end{align}


The optimal $n$-dimensional Sobolev conjugate $A_{\frac n\alpha}$, of order $\alpha \in (0,n)$, of a Young function $A$
has a central role in our results. 
 It was introduced  in \cite{Cia_CPDE} (and in \cite{Cia_IUMJ} in an equivalent form) for $\alpha =1$ and in \cite{Cia_Forum} for $\alpha \in \N$; its role in embeddings for fractional Orlicz-Sobolev spaces for non-integer $\alpha$ was discovered in \cite{ACPS1}.
  The function $A_{\frac n\alpha}$ is defined as follows.
 Assume that
\begin{equation}\label{conv-alpha}
\int _0\bigg(\frac t{A (t)}\bigg)^{\frac {\alpha}{n-\alpha}}\, dt <
\infty.
\end{equation}
Let 
$H_{A,\frac{n}{\alpha}} : [0, \infty ) \to [0, \infty)$ be the function given by 
\begin{equation}\label{H-alpha}
 H_{A,\frac{n}{\alpha}}(t)= \bigg(\int _0^t \bigg(\frac \tau{A(\tau)}\bigg)^{\frac {\alpha}{n-\alpha}}\, d\tau\bigg)^{\frac {n-\alpha}{n}} \quad \hbox{for $t \geq 0$.}
\end{equation}
Then the Sobolev conjugate $A_{\frac{n}{\alpha}}$ obeys: 
\begin{equation}\label{A-alpha}
A_{\frac{n}{\alpha}} (t)= A (H_{A,\frac{n}{\alpha}}^{-1}(t)) \quad \hbox{for $t \geq 0$.}
\end{equation}
Note that $H_{A,\frac n\alpha}$, and hence $A_{\frac n\alpha}$, are finite-valued provided that
\begin{equation}\label{divinf}
\int^\infty \bigg(\frac t{A (t)}\bigg)^{\frac \alpha{n-\alpha}}\, dt =
\infty.
\end{equation}
On the other hand, if 
\begin{equation}\label{convinf}
\int^\infty \bigg(\frac t{A (t)}\bigg)^{\frac \alpha{n-\alpha}}\, dt <
\infty,
\end{equation} 
then
\begin{align}\label{H-1inf}
    H_{A,\frac n\alpha}^{-1}(t) =\infty \quad \text{for} \quad t>  \bigg(\int _0^\infty \bigg(\frac \tau{A(\tau)}\bigg)^{\frac {\alpha}{n-\alpha}}\, d\tau\bigg)^{\frac {n-\alpha}{n}},
\end{align}
whence 
\begin{align}\label{Aninf}
A_{\frac n\alpha}(t)=\infty
\end{align}
for the same values of $t$.

Finally, we associate with a Young function $A$ the Young function $A_\diamond : [0, \infty) \to [0, \infty)$ 
 given by
\begin{equation}\label{G}
A_\diamond (t)= { A(t)^{\frac{n+1}n}} \qquad \text{for $t \geq 0$.}
\end{equation}
The Young function $A_\diamond$ arises in connection with the study of boundedness properties of the Poisson extension operator, whose definition is recalled in Section~\ref{S:4} below.


\subsection{Function spaces}\label{SS:spaces}

Let $\Omega$ be a measurable set in $\rn$.
The Orlicz space $L^A (\Omega)$, associated with a Young function $A$,
 is the Banach function
space of those   measurable functions $u : \Omega \to \mathbb R$
which render the
 Luxemburg norm
\begin{equation}\label{lux}
 \|u\|_{L^A(\Omega )}= \inf \left\{ \lambda >0 :  \int_{\Omega }A
\left( \frac{|u|}{\lambda} \right) dx  \leq 1 \right\}\,
\end{equation}
finite. In particular, $L^A (\Omega)= L^p (\Omega)$ if $A(t)= t^p$ for some
$p \in [1, \infty )$, and $L^A (\Omega)= L^\infty (\Omega)$ if $A(t)=0$ for
$t\in [0, 1]$ and $A(t) = \infty$ for $t>1$.
\\  For some steps of our proofs, we shall need to consider the functional \eqref{lux} associated with a non-decreasing and left-continuous function $A: [0, \infty) \to [0, \infty)$ which is not necessarily convex. 
If $A$ is a function of this kind, then the functional $\|u\|_{L^A(\Omega )}$ is still positively homogeneous, although it need not be a norm. The set of functions $u$ for which $\|u\|_{L^A(\Omega )}<\infty$ will be still denoted by $L^A(\Omega )$. In the rest of this section, unless otherwise stated, capital letters $A$, $B$, etc. denote Young functions. It will be explicitly mentioned if they are allowed to be just non-decreasing and left-continuous.
  \\
 The H\"older type inequality
\begin{equation}\label{holder}
\int _\Omega |u v|\,dx  \leq 2\|u\|_{L^A(\Omega )}
\|v\|_{L^{\widetilde A}(\Omega )}
\end{equation}
holds for every $u \in L^A(\Omega )$ and $v\in L^{\widetilde A}(\Omega )$.
\\ More generally,  if $A$, $B$ are non-decreasing and left-continuous functions, and $C$ is a Young function such that
\begin{equation}\label{ABC}
A^{-1}(t)B^{-1}(t) \leq k C^{-1}(t) \quad \text{for $t \geq 0$,}
\end{equation}
for some constant $k$, then 
\begin{equation}\label{holdergen}
\|uv\|_{L^C(\Omega )}  \leq 2k \|u\|_{L^A(\Omega )}
\|v\|_{L^{B}(\Omega )}
\end{equation}
for every $u \in L^A(\Omega )$ and $v \in L^B(\Omega )$.
  If the inequality \eqref{ABC} only holds for $0\leq t \leq t_0$, for some $t_0>0$, then the inequality \eqref{holdergen} holds under the additional assumption that $\|u\|_{L^\infty(\Omega)} < A^{-1}(t_0)$ and $\|v\|_{L^\infty(\Omega)} < B^{-1}(t_0)$. The inequality \eqref{holdergen} is stated in \cite{oneil} in the case when $A$, $B$ and $C$ are Young functions and \eqref{ABC} holds, with $k=1$, for every $t \geq 0$. The variants appearing here can be verified   via a close inspection of the proof of  \cite{oneil}.
\par\noindent If $A$ dominates $B$ globally, then
\begin{equation}\label{normineq}
\|u \|_{L^B(\Omega )} \leq c \|u \|_{L^A(\Omega)}
\end{equation}
for every $u \in L^A(\Omega )$, where $c$ is the same constant as in
\eqref{B.5bis}. 

Assume now that $\Omega$ is an open set and let $k\in \mathbb N$. We  set
$$C^k(\Omega) = \{  u\in C(\Omega): 
\hbox{$u$ has  continuous bounded derivatives up to the order $k$}\}.$$
Classically, $C^k(\Omega)$ is a Banach space, endowed with the norm 
$$\|u\|_{C^k(\Omega)}= \sum_{m=0}^k\|D ^m u\|_{C(\Omega)}.$$
Here, $D^mu$ denotes the vector of all partial derivatives of $u$ of order $m$. When $m=1$ we also use the simplified notation $Du$ for $D^1u$. Also, $D^0u$ stands for $u$.
\\  The notation $C^k_c(\Omega)$ is adopted  for the subspace of those functions in $C^k(\Omega)$ which are compactly supported in $\Omega$. 
\\   The space $C^k(\overline \Omega)$ consists of  the restriction to  $\overline \Omega$ of functions in $C^k(\Omega')$ for some open set $\Omega'\supset \overline \Omega$. Accordingly, $C^k_c(\overline \Omega)$ denotes the set of   functions in $C^k(\overline \Omega)$  whose support is bounded. 

The Orlicz-Sobolev space $W^{1,A}(\Omega)$ is defined as
\begin{align}\label{w1a}
W^{1,A}(\Omega) = \{  u \in L^A(\Omega) :\ 
\hbox{$u$ is  weakly differentiable and $|D u|\in L^A(\Omega)$}\},
 \end{align}
and is  a Banach space endowed with the norm
\begin{equation}\label{w1anorm}
\|u\|_{W^{1,A}(\Omega)} =  \|u \|_{L^A(\Omega)} + \|D u \|_{L^A(\Omega)}.
\end{equation}
Here, and in what follows, we use the notation $\|D u \|_{L^A(\Omega)}$ as a shorthand for $\|\, |D u|\, \|_{L^A(\Omega)}$.
\\  The homogeneous Orlicz-Sobolev space $V^{1,A}(\Omega)$ is instead defined as
\begin{align}\label{v1a}
V^{1,A}(\Omega) = \{  u:\ 
\hbox{$u$ is  weakly differentiable and $|D u|\in L^A(\Omega)$}\},
 \end{align}
and is  equipped  with the seminorm
\begin{equation}\label{v1anorm}
\|u\|_{V^{1,A}(\Omega)} =  \|D u \|_{L^A(\Omega)}.
\end{equation}
 In the case when $L^A(\Omega) =L^p(\Omega)$ for some $p \in [1, \infty)$, the definitions above reproduce the classical Sobolev space $W^{1,p}(\Omega)$ and its homogeneous counterpart 
$V^{1,p}(\Omega)$. 
\\ 
An embedding theorem for the space $V^{1,A}(\rn)$ reads as follows \cite{Cia_IUMJ, Cia_CPDE}. Assume that $A$ fulfills the condition   \eqref{conv-alpha} with $\alpha=1$ and let $A_n$ be defined by  \eqref{A-alpha} with $\alpha=1$.
Then, there exists a constant $c=c(n)$ such that
\begin{equation}\label{emb}
 \|u \|_{L^{A_n}(\rn )} \leq c\|\nabla u \|_{L^A(\rn )}
\end{equation}
for every function $u\in V^{1,A}(\rn)$ such that $|\{|u|>t\}|<\infty$ for every $t>0$. Hence,  the inequality \eqref{emb} holds for every $u \in W^{1,A}(\rn)$.
\\
In particular, if $A$ grows so fast near infinity for the condition \eqref{convinf} to be satisfied, then, owing to \eqref{Aninf},
\begin{equation}\label{embinf}
 \|u \|_{L^{\infty}(\rn )} \leq c\|\nabla u \|_{L^A(\rn )}
\end{equation}
for some constant $c=c(n,A)$ and every function $u$ as in \eqref{emb}. Moreover, \begin{align}\label{cont}
    u\in C(\rn).
\end{align}
 We  also need to introduce Orlicz-Sobolev  spaces of functions on $\rnp$, which decay near infinity, and whose trace on 
 $\partial \rnp$ belongs to a given Orlicz space.
 To this purpose, we begin by setting 
\begin{align}\label{abuse}
\overset{\circ} V {}^{1,A}( \rnp) = \big\{u_{{\big| \rnp}}:& \,\text{$u$ belongs to the closure of $C^\infty_c(\rnn)$} \\ \nonumber  & \text{ with respect to the seminorm in $V^{1,A}(\rnn)$}\big\}.
\end{align}
In particular, the trace   $\tr  \in L^1_{\rm loc} (\rn)$ over $\rn$ of every function $u \in \overset{\circ} V {}^{1,A}(\rnp)$ is well defined for every Young function $A$. This is a consequence of the fact that  $u \in V^{1,A}(\mathbb R^{n+1})\subset V^{1,1}_{\rm loc}(\mathbb R^{n+1})$ and $\tr $ is well defined as a function in $ L^1_{\rm loc} ( \rn)$ for any $u \in V^{1,1}_{\rm loc}(\mathbb R^{n+1})$. Notice that, since
$$\partial \rnp = \rn \times \{0\} \approx \rn,$$
here, and in what follows, 
    $\partial \rnp$ is identified with   $\rn$.
\\  Given two Young functions $A$ and $B$, we define the space
\begin{equation}\label{V1AB}
 V^{1,(A,B)}(\rnp, \rn)=   \big\{u \in \overset{\circ} V {}^{1,A}( \rnp):   \,\,\tr \in L^B  (\rn)\big\}
\end{equation}
endowed with the norm
\begin{equation}\label{normV1AB}
\|u\|_{  V ^{1,(A,B)}(\rnp, \rn)}=    \|\nabla u\|_{L^A(\rnp)}+ \|\tr \|_{L^B(\rn)}.
\end{equation}

Higher-order 
Orlicz-Sobolev spaces  come into play to deal with 
 polyharmonic problems.  They require the use of spaces defined in terms of $k$-th order gradients of a function $u$ defined as
\begin{equation}\label{higher-der}
\nabla^k u=\begin{cases}
\Delta^{k/2}u &\mbox{if } k \mbox{ is even}  \\
\nabla\Delta^{\frac{k-1}{2}} u &\mbox{if } k \mbox{ is odd}.
\end{cases}    
\end{equation}
Notice that  $\nabla^k u$ differs from $D^ku$, as defined above.  In particular, 
 $\nabla^k u \in \R$ if $k$ is even, whereas $\nabla^k u \in \R^n$ if $k$ is odd.
\\
The homogeneous Orlicz-Sobolev space $V^{k,A}(\Omega)$ is defined as
\begin{align}\label{vka}
V^{k,A}(\Omega) = \{  u:\ 
\hbox{$\nabla^{k}u$ exists in the weak sense and $|\nabla^k u|\in L^A(\Omega)$}\},
 \end{align}
and is  equipped  with the seminorm
\begin{equation}\label{vkanorm}
\|u\|_{V^{k,A}(\Omega)} =  \|\nabla^k u \|_{L^A(\Omega)}.
\end{equation}
 In analogy with \eqref{abuse},
 we set
 \begin{align}\label{kabuse}
\overset{\circ} V {}^{k,A}( \rnp) = \big\{u_{{\big| \rnp}}:& \,\text{$u$ belongs to the closure of $C^\infty_0(\rnn)$} \\ & \nonumber \text{ with respect to the seminorm in $V^{k,A}(\rnn)$}\big\}.
\end{align}
A similar argument as in the first-order case shows that the trace   $\tr  \in L^1_{\rm loc} (\rn)$  of every function $u \in \overset{\circ} V {}^{k,A}(\rnp)$ is well defined for every $k \in \N$ and every Young function $A$.
\\ Furthermore,
given two Young functions $A$ and $B$, we define the space
\begin{equation}\label{VkAB}
 V^{k,(A,B)}(\rnp, \rn)=   \big\{u \in \overset{\circ} V {}^{k,A}( \rnp):   \,\,\tr \in L^B  (\rn)\big\}
\end{equation}
endowed with the norm
\begin{equation}\label{normVkAB}
\|u\|_{  V ^{k,(A,B)}(\rnp, \rn)}=    \|\nabla^k u\|_{L^A(\rnp)}+ \|\tr \|_{L^B(\rn)}.
\end{equation}
  In what follows, when there is no ambiguity,  the value of the trace  $\tr$ of a function $u$ defined in $\rnp$ at a point $x \in \rn$ will simply be denoted by $u(x,0)$.

\section{Main results}\label{S:3}

Our main results admit a unified formulation for  problems \eqref{eq:gen-eq} of arbitrary order $m$. Because of the importance of second-order problems and of the ease of the corresponding notation, we enucleate the pertaining statements in  Subsection \ref{S:second}. The higher-order case is discussed in the subsequent Subsection \ref{S:higher}.

\subsection{Harmonic functions}\label{S:second}

In order to formulate the assumptions in  our  main results we need to introduce some notation and definitions.
\\
Let $\mathcal N: \mathbb R \rightarrow \mathbb R$ be the function appearing in  the problem~\eqref{eq:main-eq}. We  
assume that
there exists a finite-valued Young function $A$ such that 
\begin{equation}\label{eq:N1}
 |\mathcal N(t)| \leq A(|t|^{n'})  \quad \text{for $t \in \R$,}
\end{equation}
and 
\begin{equation}\label{eq:N2}
|\mathcal N(t)-\mathcal N(s)|\leq c|t-s|\bigg(\frac{A( |\theta t|^{n'})}{|t|} +\frac{A( |\theta s|^{n'})} {|s|}\bigg) \quad  \text{for $t ,s \in \R$,}
\end{equation}
for some positive constants $c$ and $\theta$. Here, and in similar occurrences in what follows, the function $\frac{A(t^{n'})}{t}$ is extended as $0$ by continuity for $t=0$. 
For simplicity of notation, we denote this function by $D$. Namely,  $D: [0,\infty) \rightarrow [0,\infty)$ is given by 
\begin{equation}\label{D}
 D(t)= \begin{cases}\displaystyle \frac{A(t^{n'})}{t} & \quad  \text{if $t>0$}
\\
0 & \quad  \text{if $t=0$.}
\end{cases}
\end{equation}
Next, 
define the   finite-valued Young function $E$ as 
\begin{equation}\label{E}
E(t) = \int_0^{A^{-1}(t)} a(s)^n\, ds  \qquad \text{for $t \geq 0$,}
\end{equation}
where $a$ denotes the function from equation \eqref{young}. 
The conditions   \eqref{conv-alpha} and \eqref{divinf} are fulfilled with  $\alpha=1$ and   $A$ replaced with $E$. This is proved in  Lemma \ref{Enequiv}, Section \ref{S:aux}. In that lemma, the Sobolev conjugate  $E_n$ of $E$, defined as in   \eqref{A-alpha} with $\alpha=1$, is also shown to obey:
\begin{equation}\label{nov1}
E_n(t) \simeq  \int_0^{t^{n'}} a(s)^n\, ds  \qquad \text{for $t \geq 0$.}
\end{equation}
Finally, let $E_\diamond$ 
 be the Young function associated with $E$ as in~\eqref{G}. Namely,
 \begin{align}
     \label{2024-260}
     E_\diamond (t) = E(t)^{\frac{n+1}{n}} \qquad \text{for $t \geq 0$.}
 \end{align}
%
%
%
\\ The functions $E_\diamond$ and $E_n$ provide us with the appropriate Orlicz-Sobolev-trace ambient space for weak solutions to the problem \eqref{eq:main-eq}.

\begin{defn}\label{def:weak-sol1} {\rm\bf{ [Weak solutions to problem \eqref{eq:main-eq}]}} 
Assume that $f\in L^{1}_{loc}(\rn)$ and  the condition \eqref{eq:N1} is fulfilled for some Young function $A$. Let $E$, $E_n$ and $E_\diamond$ be the functions given by \eqref{E}, \eqref{nov1}, and \eqref{2024-260}. We say that a function $u\in V^{1,(E_\diamond,E_n)}({\rnp}, \rn)$ is a weak solution to the problem \eqref{eq:main-eq} if  
\begin{align*}
\int_{\rnp}\nabla u\cdot\nabla \varphi \, dxdx_{n+1}- \int_{\rn}(\mathcal N(u)(x,0)+f(x))\varphi(x,0)\, dx=0
\end{align*}
for every $\varphi\in C^{\infty}_c\big(\overline{\rnp}\big)$.
\end{defn}

Our existence result for weak solutions to the problem \eqref{eq:main-eq} reads as follows.
 
\begin{thm}\label{thm1} {\rm\bf{ [Existence of weak solutions to problem \eqref{eq:main-eq}]}}
Assume that   $\mathcal N$ satisfies the assumptions  \eqref{eq:N1}  and \eqref{eq:N2} for some Young function $A$ such that
\begin{equation}\label{A} \tfrac 1{c}a(A^{-1}(2t)) \leq a(A^{-1}(t))\leq \tfrac 12 a(A^{-1}(ct)) \,\,\, \text{for \,$t \geq 0$,}
\end{equation}
and for some constant $c>0$.  Let $E$, $E_n$ and $E_\diamond$ be the functions given by \eqref{E}, \eqref{nov1}, and \eqref{2024-260}. Then, there exist  constants $\sigma _0$ and $c_0$ such that if $$\|f\|_{L^{E}(\rn)}\leq \sigma_0,$$ then the problem \eqref{eq:main-eq} admits a weak  solution 
$u \in C^\infty(\rnp) \cap V^{1,(E_\diamond,E_n)}({\rnp}, \rn)$ fulfilling 
\begin{equation}\label{nov4}
 \|u\|_{V^{1,(E_\diamond ,E_n)}({\rnp},  \rn)} \leq \sigma_0 c_0.
\end{equation}
Assume, in addition, that 
$\mathcal N(t)>0$ for $t>0$.  If $f >0$  a.e., then $u>0$ as well.
\end{thm}
%

The next theorem  amounts to a Sobolev regularity result for the solution $u$ to the problem \eqref{eq:main-eq} under a stronger integrability assumption on the datum $f$. It tells us that the membership of $f$ in an Orlicz space whose norm is stronger, in a qualified sense, than $L^E(\rn)$ 
is reflected into
 stronger integrability properties of the derivatives of $u$. 
\\
The Orlicz ambient space for the datum $f$ is described via a Young function 
 $F$ such that 
\begin{equation}\label{E:overlineE}
\int_{0} \left(\frac{t}{F(t)}\right)^{\frac{1}{n-1}}\,dt<\infty. 
\end{equation}
The relevant Orlicz space is associated with the Young function $E\vee F$, given by
$$(E\vee F) (t) =\max\{E(t), F(t)\} \qquad \text{for $t\geq 0$.}$$
Note that
$$L^{E\vee F}(\rn)= L^E(\rn) \cap L^F(\rn).$$
The   Sobolev conjugate $F_n$ of $F$, defined as in  \eqref{A-alpha} with $\alpha=1$, and the function $F_\diamond$ associated with $F$ as in~\eqref{G}, also  play a role. 

\begin{thm}\label{thm2}{\rm\bf{ [Sobolev regularity of solutions to problem \eqref{eq:main-eq}]}}
Assume that $A$, $E$ and $\sigma _0$ are as in Theorem~\ref{thm1}. Let $F$ be a Young function satisfying~\eqref{E:overlineE} and such that $F \in \Delta_2 \cap \nabla_2$. Assume  that there exists  a constant $\kappa $ such that
\begin{equation}\label{H}
\frac{F_n^{-1}(t)}{F^{-1}(t)} \leq \kappa \frac{E_n^{-1}(t)}{E^{-1}(t)} \quad \text{for} \quad t>0.
\end{equation}
Let $f\in L^{E\vee F}(\rn) $. There exists a constant  $\sigma_1 \in (0,\sigma_0)$ such that, if 
$\|f\|_{L^E(\rn)} <\sigma_1$, 
then the solution $u$ to the problem~\eqref{eq:main-eq} provided by Theorem~\ref{thm1} satisfies
\begin{equation}\label{nov6}
u \in  V^{1,(F_\diamond,F_n)}({\rnp}, \rn).
\end{equation}
\end{thm}


\begin{rem}{\rm
If  $F$ is a Young function whose Matuszewska-Orlicz index satisfies $I(F)<n$, then the inequality~\eqref{H} holds whatever $E$ is. This is a consequence of Lemma \ref{index-lemma} from  Section \ref{S:aux}.}
\end{rem}

Our last main result about the problem \eqref{eq:main-eq} proposes additional assumptions 
 on the data $\mathcal N$ and $f$ for the  solution discussed so far to be classical, 
namely smooth in $\rnp$ and continuously differentiable up to its boundary.
\\ The function $\mathcal N$ is required to be differentiable, at least for small values of its argument, with a derivative whose modulus satisfies an appropriate growth condition.
\\ The function $f$ is assumed to belong to an Orlicz-Sobolev space defined by some Young function $M$
  which agrees with $E$ for small values of its argument, and is suitably modified for large values. The function  $M$ has to be  finite-valued and obey:
\begin{equation}\label{M}
\begin{cases}
  M(t) = E(t)\,\,\, \text{near $0$,}
  \\ \\
\displaystyle \int ^{\infty}\bigg(\dfrac{t}{M(t)}\bigg)^{\frac{1}{n-1}}dt<\infty.  
\end{cases}
\end{equation}
Observe that  the second condition in \eqref{M} implies that the Sobolev conjugate $M_n$ of $M$ fulfils:
\begin{align}\label{aug1}
M_n(t) = \infty \quad \text{for large $t$.}
\end{align}

\begin{thm}\label{thm:reg2}{\rm\bf{ [Smooth solutions to problem \eqref{eq:main-eq}]}}
Let $A$  be as in Theorem~\ref{thm1}, let $D$ be the function defined by \eqref{D}, and let $M$ be a Young function satisfying \eqref{M}.
Assume that:
\\ (i) The function $D$ admits the lower bound:
 \begin{align}
    \label{2024-160}
     D(t)\geq  \phi(t) \psi (t) \quad \text{for \,$t > 0$,}
\end{align}
for some non-decreasing continuous functions $\phi, \psi : (0, \infty) \to (0, \infty)$  such that 
\begin{equation}\label{E:convexity}
\phi(\varepsilon t) \leq   c_0\, \varepsilon^\gamma \phi\left(t\right) \quad \text{for $\varepsilon \in (0,1)$   and $t\in (0, t_0)$},
\end{equation}
 for some positive constants  $\gamma$, $c_0$ and $t_0$.
 \\ (ii)
 The function $\mathcal N$ is differentiable and
\begin{equation}\label{eq:N3}
|\mathcal N'(t)-\mathcal N'(s)| \leq c \,\phi(\theta|t-s|)\, (\psi(\theta |t|)+\psi(\theta |s|)) \,\,\, \text{for $t$ and $s$ near 0,}
\end{equation}
for some positive constants $\theta$ and $c$.
\\ Then, there exists a constant $\eta>0$ such that, if $f\in W^{1,M}(\rn)$ and $\|f\|_{W^{1,M}(\rn)}\leq \eta$, then the weak solution $u$ to problem~\eqref{eq:main-eq} provided by Theorem~\ref{thm1} is  a classical solution, in the sense that
\begin{equation}\label{nov10}
u \in  C^{\infty}(\rnp)\cap C^1\big(\overline{\rnp}\big).
\end{equation}
\end{thm}

\begin{rem}\label{asympt}
{\rm Plainly, only the asymptotic behaviour of the functions $\phi$ and $\psi$ in Theorem \ref{thm:reg2} is relevant.
}
\end{rem}

\begin{rem}
{\rm The condition \eqref{E:convexity} can be equivalently formulated by requiring that the index $i_0(\phi)$, defined as in \eqref{index0}, satisfies:
\begin{align}
   \label{2024-8}
   i_0(\phi)>0.
\end{align}
}
\end{rem}

 \begin{rem}{\rm
In analogy with the inequality~\eqref{eq:N2}, one may be tempted to simplify the assumption~\eqref{eq:N3} as
\[
|\mathcal N'(t)-\mathcal N'(s)| \leq c |t-s| \left(\frac{D(\theta |t|)}{|t|}+\frac{D(\theta |s|)}{|s|}\right).
\]
However, the latter condition is informative only when $D(t)/t$ is non-decreasing, a property which is not fulfilled 
for various important 
choices of the function $\mathcal N$. 
By contrast, the assumption \eqref{eq:N3} is sufficiently general  for Theorem \ref{thm:reg2}
to be applied in most situations of interest.
\\ In particular,  if $D(t)/t$ is non-increasing, the condition \eqref{eq:N3}
is typically satisfied with 
$\phi(t)=D(t)$ and $\psi(t)=1$ in~\eqref{2024-160}. Notice that the assumption~\eqref{E:convexity}  then holds with $\gamma=n'-1$. Indeed, by the convexity of $A$, if $\varepsilon \in (0,1)$ then  
\begin{equation}\label{E:D-epsilon}
\frac{D(t)}{\varepsilon^{n'-1}}=\frac{A(t^{n'})}{\varepsilon^{n'} \cdot \frac{t}{\varepsilon}}
\leq \frac{A\left(\left(\frac{t}{\varepsilon}\right)^{n'}\right)}{\frac{t}{\varepsilon}}
=D(t/\varepsilon) \quad \text{for $t>0$}.
\end{equation}
Specific  choices of the functions $\phi$ and $\psi$  in~\eqref{2024-160} are described in the  examples below.}
\end{rem}

\begin{exam}\label{ex1}{\rm
Assume that $\mathcal N\in C^1(\mathbb R)$ and
\begin{equation}\label{oct1}
\mathcal N(t)  = \begin{cases}   |t|^{p_0-1} t & \quad \text{for $t$ near $0$}
\\
|t|^{p_1-1} t & \quad \text{for $|t|$ near $\infty$,}
\end{cases}
\end{equation}
for some $p_0, p_1 >n'$. Then, the conditions \eqref{eq:N1} and \eqref{eq:N2} are fulfilled with 
\begin{equation}\label{oct2}
A(t)  = \begin{cases}   t^{\frac{p_0}{n'}} & \quad \text{for $t$ near $0$}
\\
t^{\frac{p_1}{n'}}  & \quad \text{for $t$ near $\infty$.}
\end{cases}
\end{equation}
Hence, from equations \eqref{E} and \eqref{nov1} one infers that
\begin{equation}\label{2024-400}
 E(t)  \simeq \begin{cases}   t^{\frac{n}{p_0'}} & \quad \text{for $t$ near $0$}
\\
t^{\frac{n}{p_1'}}  & \quad \text{for $t$ near $\infty$,}
\end{cases}
\end{equation}
\begin{equation}\label{oct3}
 E_\diamond(t)  \simeq \begin{cases}   t^{\frac{n+1}{p_0'}} & \quad \text{for $t$ near $0$}
\\
t^{\frac{n+1}{p_1'}}  & \quad \text{for $t$ near $\infty$,}
\end{cases}
\end{equation}
and
\begin{equation}\label{oct4}
E_n(t)  \simeq \begin{cases}   t^{n(p_0-1)} & \quad \text{for $t$ near $0$}
\\
t^{n(p_1-1)}  & \quad \text{for $t$ near $\infty$.}
\end{cases}
\end{equation}
 Moreover, if $p_0 \geq 2$, then the  bound in~\eqref{2024-160} holds with
\[
\phi(t) \simeq t \quad \text{and} \quad \psi(t)\simeq t^{p_0-2} \quad \text{for $t$ near $0$},
\]
whereas, for $p_0 \leq 2$, an admissible choice is
\[
\phi(t)\simeq t^{p_0-1} \quad \text{and} \quad \psi(t)\simeq 1 \quad \text{for $t$ near $0$}.
\]}
\end{exam}

\begin{exam}\label{ex2}{\rm
Assume that $\mathcal N\in C^1(\mathbb R)$ and
\begin{equation}\label{oct5}
\mathcal N(t)  = \begin{cases}   |t|^{p_0-1} t (\log 1/|t|)^{\beta_0}& \quad \text{for $t$ near $0$}
\\
|t|^{p_1-1} t (\log |t|)^{\beta_1}& \quad \text{for $|t|$ near $\infty$,}
\end{cases}
\end{equation}
for some $p_0, p_1 >n'$ and $\beta_0, \beta_1 \in \mathbb R$. Then, the conditions \eqref{eq:N1} and \eqref{eq:N2} are fulfilled with 
\begin{equation}\label{oct6}
A(t)  \simeq \begin{cases}   t^{\frac{p_0}{n'}} (\log 1/t)^{\beta_0} & \quad \text{for $t$ near $0$}
\\
t^{\frac{p_1}{n'}}  (\log t)^{\beta_1} & \quad \text{for $t$ near $\infty$.}
\end{cases}
\end{equation}
Hence,
\begin{equation}\label{2024-401}
 E(t)  \simeq \begin{cases}   t^{\frac{n}{p_0'}} (\log 1/t)^{\frac {n\beta_0}{p_0}}  & \quad \text{for $t$ near $0$}
\\
t^{\frac{n}{p_1'}}  (\log t)^{\frac{n\beta_1}{p_1}} & \quad \text{for $t$ near $\infty$,}
\end{cases}
\end{equation}
\begin{equation}\label{oct7}
 E_\diamond(t)  \simeq \begin{cases}   t^{\frac{n+1}{p_0'}} (\log 1/t)^{\frac {\beta_0}{p_0}(n+1)}  & \quad \text{for $t$ near $0$}
\\
t^{\frac{n+1}{p_1'}}  (\log t)^{\frac{\beta_1}{p_1}(n+1)} & \quad \text{for $t$ near $\infty$,}
\end{cases}
\end{equation}
and
\begin{equation}\label{oct8}
E_n(t)  \simeq \begin{cases}   t^{n(p_0-1)} (\log 1/t)^{n\beta_0} & \quad \text{for $t$ near $0$}
\\
t^{n(p_1-1)}   (\log t)^{n\beta_1}  & \quad \text{for $t$ near $\infty$.}
\end{cases}
\end{equation}
 As for the condition \eqref{2024-160}, if either $p_0 > 2$, or $p_0=2$ and $\beta \leq 0$, then it holds with
\[
\phi(t) \simeq t \quad \text{and} \quad \psi(t) \simeq t^{p_0-2} (\log 1/t)^{\beta_0} \quad \text{for $t$ near $0$},
\]
whereas, if either  $p_0 > 2$, or $p_0=2$ and $\beta \geq 0$, then \eqref{2024-160} is fulfilled with 
\[
\phi(t) \simeq t^{p_0-1}(\log 1/t)^{\beta_0} \quad \text{and} \quad \psi(t) \simeq 1 \quad \text{for $t$ near $0$}.
\]}
\end{exam}

\begin{exam}\label{ex3}{\rm
Assume that  
\begin{equation}\label{oct9}
\mathcal N(t)  =   |t|^{p-1} t\, e^{t^\alpha}
\end{equation}
for some $p >n'$ and $\alpha  >0$. Then, the conditions \eqref{eq:N1} and \eqref{eq:N2} are fulfilled with 
\begin{equation}\label{oct10}
A(t)   \simeq  t^{\frac p{n'}} e^{t^{\frac \alpha{n'}}}.
\end{equation}
Therefore,
\begin{equation}\label{2024-402}
E(t)  \simeq \begin{cases}   t^{\frac{n}{p'}}  & \quad \text{for $t$ near $0$}
\\
t^{n}  (\log t)^{n-1- \frac n\alpha} & \quad \text{for $t$ near $\infty$,}
\end{cases}
\end{equation}
\begin{equation}\label{oct11}
E_\diamond(t)  \simeq \begin{cases}   t^{\frac{n+1}{p'}}  & \quad \text{for $t$ near $0$}
\\
t^{n+1}  (\log t)^{(n+1)(\frac 1{n'}-\frac{1}\alpha)} & \quad \text{for $t$ near $\infty$}
\end{cases}
\end{equation}
and
\begin{equation}\label{oct12}
E_n(t)  \simeq
\begin{cases}   t^{n(p-1)}   & \quad \text{for $t$ near $0$}
\\
e^{t^\alpha}  & \quad \text{for $t$ near $\infty$.}
\end{cases}
\end{equation}
 If $p \geq 2$, then the inequality ~\eqref{2024-160} is satisfied with
\[
\phi(t) \simeq t \quad \text{and} \quad \psi(t) \simeq t^{p-2}  \quad \text{for $t$ near $0$};
\]
if $p \leq 2$, it instead holds with
\[
\phi(t) \simeq t^{p-1} \quad \text{and} \quad \psi(t) \simeq 1 \quad \text{for $t$ near $0$}.
\]}
\end{exam}

\subsection{Polyharmonic functions}\label{S:higher}
The results from the previous subsection admit analogs for the polyharmonic problem 
~\eqref{eq:gen-eq} of arbitrary order $2m$, with $m\in \N$. Their statements make use of functions $A$, $D$ and $E$, depending on $m$, and having the same role as those introduced for $m=1$.  
\\
Set \begin{align}\label{ell}
    \ell=2m-1,
\end{align}
and let $\mathcal N: \mathbb R \rightarrow \mathbb R$ be as in~\eqref{eq:gen-eq}. We
assume that
there exists a   finite-valued Young function $A$ such that 
\begin{equation}\label{eq:N1new}
 |\mathcal N(t)| \leq A(|t|^{\frac{n}{n-\ell}})  \quad \text{for $t \in \R$,}
\end{equation}
and 
\begin{equation}\label{eq:N2new}
|\mathcal N(t)-\mathcal N(s)|\leq c|t-s|\bigg(\frac{A( |\theta t|^{\frac{n}{n-\ell}})}{|t|} +\frac{A( |\theta s|^{\frac{n}{n-\ell}})} {|s|}\bigg) \quad  \text{for $t ,s \in \R$,}
\end{equation}
for some positive constants $c$ and $\theta>0$. The function $D$ is accordingly defined as
\begin{equation}\label{Dk}
 D(t)= \begin{cases}\displaystyle \frac{A(t^{\frac{n}{n-\ell}})}{t} & \quad  \text{if $t>0$}
\\
0 & \quad  \text{if $t=0$,}
\end{cases}
\end{equation}
and the 
 Young function $E: [0, \infty) \to [0, \infty)$ as 
\begin{equation}\label{Enew}
E(t) = \int_0^{A^{-1}(t)} a(s)^{\frac{n}{\ell}}\, ds  \qquad \text{for $t \geq 0$,}
\end{equation}
where $a$ denotes the function appearing in \eqref{young}. 
Lemma \ref{Enequiv} ensures that the condition \eqref{conv-alpha} 
is fulfilled with $A$ replaced with $E$, and its Sobolev conjugate $E_{\frac{n}{\ell}}$, defined as in \eqref{A-alpha}, obeys
\begin{equation}\label{nalpha}
E_{\frac{n}{\ell}}(t) \simeq  \int_0^{t^{\frac{n}{n-\ell}}} a(s)^{\frac{n}{\ell}}\, ds  \qquad \text{for $t \geq 0$.}
\end{equation}
The Young function
$E_\diamond$ 
 is defined via $E$ as in~\eqref{G}.

\begin{defn}\label{def:weak-sol}  {\rm\bf{ [Weak solutions to problem \eqref{eq:gen-eq}]}}   Let $m\geq 2$ and let $\ell$ be given by \eqref{ell}. Assume that $f\in L^{1}_{loc}(\rn)$ 
and that the condition \eqref{eq:N1new} is satisfied for some Young function $A$. 
Let $E$, $E_{\frac{n}{\ell}}$ and $E_\diamond$ be the functions  given by \eqref{Enew}, \eqref{nalpha}, and \eqref{2024-260}.
We say that $u\in V^{\ell,(E_\diamond,E_{\frac{n}{\ell}})}({\rnp}, \rn)$ is a weak solution to the problem \eqref{eq:gen-eq} if  
\begin{align*}
\int_{\rnp}\nabla^m u\cdot \nabla^m \varphi \,dxdx_{n+1}-\int_{\rn}(\mathcal N(u)(x,0)+f(x))\varphi(x,0)\, dx=0
\end{align*}
for every $\varphi\in C^{\infty}_c\big(\overline{\rnp}\big)$ such that $(\partial_{x_{n+1}}\Delta^{k}\varphi)(\cdot,0)=0$ for all $k=0,...,m-2$.
\end{defn}
The existence result for solutions to the problem ~\eqref{eq:gen-eq} is the subject of the following theorem.

 \begin{thm}\label{thm:g}{\rm\bf{ [Existence of weak solutions to problem \eqref{eq:gen-eq}]}}   Let $m\geq 2$ and let $\ell$ be given by \eqref{ell}.
 Assume that the function $\mathcal N$ satisfies the assumptions  \eqref{eq:N1new}  and \eqref{eq:N2new} for some Young function $A$ fulfilling \eqref{A}. Let $E$, $E_{\frac{n}{\ell}}$ and $E_\diamond$ be the functions  given by \eqref{Enew}, \eqref{nalpha}, and \eqref{2024-260}.
 %
Then, there exist positive constants $\sigma _0$ and $c$ such that, if $$\|f\|_{L^{E}(\rn)}\leq \sigma_0,$$ 
then the problem \eqref{eq:gen-eq} admits a weak solution 
$u \in C^\infty(\rnp) \cap V^{\ell,( E_\diamond,E_{\frac{n}{\ell}})}(\rnp, \rn)$ satisfying 
\begin{equation}\label{nov4k}
 \|u\|_{V^{\ell,( E_\diamond ,E_{\frac{n}{\ell}})}(\rnp,  \rn)} \leq \sigma_0 c.
\end{equation}
Assume, in addition, that 
$\mathcal N(t)>0$ for $t>0$.  If $f >0$   a.e., then $u>0$ as well.
\end{thm}

The Sobolev regularity of the solution to the problem \eqref{eq:gen-eq} calls into play 
 a Young function $F$ such that 
\begin{equation}\label{F}
\int_{0} \left(\frac{t}{F(t)}\right)^{\frac{\ell}{n-\ell}}\,dt<\infty,
\end{equation}
its Sobolev conjugate $F_{\frac n \ell}$ defined as in  \eqref{A-alpha}, and the Young function $F_\diamond$ associated with $F$ as in~\eqref{G}. 

\begin{thm}\label{thm4}{\rm\bf{ [Sobolev regularity of solutions to problem \eqref{eq:gen-eq}]}}   Let $m\geq 2$ and let $\ell$ be given by \eqref{ell}.
 Assume that $A$, $E$ and $\sigma _0$ are as in Theorem~\ref{thm:g}. Let $F$ be a Young function satisfying \eqref{F} and such that $F \in \Delta_2 \cap \nabla_2$. Assume that there exists a constant $\kappa$ such that 
 \begin{equation}\label{HNew}
\frac{F_{\frac n \ell}^{-1}(t)}{F^{-1}(t)} \leq \kappa \frac{E_{\frac n \ell}^{-1}(t)}{E^{-1}(t)} \quad \text{for $t>0$.}
\end{equation}
Let $f\in L^{E\vee F}(\rn) $. Then, there exists a constant  $\sigma_1 \in (0,\sigma_0)$ such that, if $\|f\|_{L^E(\rn)} <\sigma_1$, 
then the solution $u$ to the problem~\eqref{eq:gen-eq} provided by Theorem~\ref{thm:g} satisfies
\begin{equation}\label{2024-11}
u \in  V^{\ell,(F_\diamond,F_{\frac n \ell})}({\rnp}, \rn).
\end{equation}
\end{thm}

 We conclude by focusing on the boundary regularity of solutions to the problem   \eqref{eq:gen-eq}.
 To this purpose, consider a    Young function $M$ as in \eqref{M}. Notice that the convergence of the integral in \eqref{M} implies that
\begin{equation}\label{Mk}
\int ^{\infty}\bigg(\dfrac{t}{M(t)}\bigg)^{\frac{\ell}{n-\ell}}dt<\infty.    
\end{equation}

\begin{thm}\label{thm:gen-eq-reg}{\rm\bf{ [Smooth solutions to problem \eqref{eq:gen-eq}]}}
  Let $m\geq 2$ and let $\ell$ be given by \eqref{ell}.  Let $A$  be as in Theorem~\ref{thm:g}, and let $D$ be the function defined by \eqref{Dk}.  
Assume that:
\\ (i) The function $D$ admits the lower bound \eqref{2024-160},
with $\phi$ fulfilling the condition \eqref{E:convexity}.
\\ (ii)  
 The function $\mathcal N$ is differentiable and satisfies the condition \eqref{eq:N3}.
\\ Then, there exists a constant $\eta>0$ such that, if $f\in W^{1,M}(\rn)$ and $\|f\|_{W^{1,M}(\rn)}\leq \eta$, then the weak solution $u$ to the problem \eqref{eq:gen-eq} provided by Theorem~\ref{thm:g} is  a classical solution, in the sense that
\begin{equation}\label{nov10k}
u \in  C^{\infty}(\rnp)\cap C^\ell(\overline{\rnp}).
\end{equation}
\end{thm}

\section{Auxiliary results on Young functions and their conjugates}\label{S:aux}

This section is devoted to some specific properties of Young functions in connection with their Sobolev conjugates, under the assumptions imposed in our main results.
Throughout the section, we assume that $\alpha\in (0,n)$.

\begin{lem}
\label{index-lemma}
Let $A$ be a finite-valued Young function fulfilling the condition \eqref{conv-alpha} and let $A_{\frac n\alpha}$ be its Sobolev conjugate given by \eqref{A-alpha}. Then
\begin{align}
    \label{2024-1}
    \frac{A^{-1}(t)}{A_{\frac n \alpha}^{-1}(t)}\leq t^{\frac \alpha n}\qquad \text{for $t >0$.}
\end{align}
Assume, in addition, that 
\begin{align}
    \label{2024-2}
    I(A)<{\frac n \alpha}.
\end{align}
Then, there exists a constant $c$ such that
\begin{align}
    \label{2024-3}
    \frac{A^{-1}(t)}{A_{\frac n \alpha}^{-1}(t)}\geq c \,t^{\frac \alpha n}\qquad \text{for $t >0$.}
\end{align}
\end{lem}
\begin{proof}
 The property \eqref{2024-5} implies that
 \begin{align}
     \label{2024-4}
     H_{A,\frac n \alpha}(t)& =  \bigg(\int _0^t \bigg(\frac \tau{A(\tau)}\bigg)^{\frac \alpha {n-\alpha}}\, d\tau\bigg)^{\frac {n-\alpha}{n}} \geq \bigg(\frac t{A(t)}\bigg)^{\frac \alpha n}\bigg(\int_0^t d\tau\bigg)^{\frac {n-\alpha}{n}}= \frac t{A(t)^{\frac \alpha n}} \quad \text{for $t>0$.}
 \end{align}
 Hence, 
 $$t\leq A(t)^{\frac \alpha n} H_{A,\frac n \alpha}(t) \quad \text{for $t>0$,}$$
 an equivalent form of \eqref{2024-1}.
\\ Next, observe that the assumption \eqref{2024-2}
is equivalent to the fact that for every $\varepsilon>0$ there exists a constant $c$ such that
\begin{align}
    \label{2024-6}
    \frac{A(t)}{t^{{\frac n \alpha}-\varepsilon}}\leq c \frac{A(\tau)}{\tau^{{\frac n \alpha}-\varepsilon}} \quad \text{if $t \geq \tau$.}
\end{align}
Thereby,
 \begin{align}
     \label{2024-7}
     H_{A,\frac n \alpha}(t)& =  \bigg(\int _0^t \bigg(\frac \tau{A(\tau)}\bigg)^{\frac \alpha {n-\alpha}}\, d\tau\bigg)^{\frac {n-\alpha}{n}}  = 
      \bigg(\int _0^t \bigg(\frac {\tau^{{\frac n \alpha}-\varepsilon}}{A(\tau)}\bigg)^{\frac \alpha {n-\alpha}}\tau^{-1+\frac {\varepsilon \alpha}{n-\alpha}}\, d\tau\bigg)^{\frac {n-\alpha}{n}} 
      \\ \nonumber & \leq c \frac{t^{1-\frac {\varepsilon \alpha} n}}{A(t)^{\frac \alpha n}}
 \bigg(\int _0^t \tau^{-1+\frac {\varepsilon \alpha}{n-\alpha}}\, d\tau\bigg)^{\frac {n-\alpha}{n}} = c' \frac{t}{A(t)^{\frac \alpha n}}
     \quad \text{for $t>0$,}
 \end{align}
for suitable constants $c$ and $c'$. Consequently,
$$c't\geq A(t)^{\frac \alpha n} H_{A,\frac n \alpha}(t) \quad \text{for $t>0$,}$$
whence \eqref{2024-3} follows.
\end{proof}

The statement of the next lemma involves the notion of quasi non-decreasing function. This means that the relevant function fulfills the definition of non-decreasing monotonicity, up to a multiplicative constant.

\begin{lem}\label{quotient}
Let $A$ be a finite-valued Young function fulfilling the condition \eqref{conv-alpha}.
\\   (i) Assume that $A\in \nabla_2$ globally.
Then, there exists $\delta >0$ such that the function 
\begin{equation}\label{nov20}
\dfrac{A_{\frac{n}{\alpha}}(t)}{t^{\frac{n}{n-\alpha}+\delta}} \quad \text{is non-decreasing in $(0, \infty)$.}
\end{equation}
(ii) Assume that $A\in \nabla_2$ near $0$, and satisfies the condition \eqref{convinf}.
Then, there exists $\delta >0$ such that the function 
\begin{equation}\label{nov20bis}
\dfrac{A_{\frac{n}{\alpha}}(t)}{t^{\frac{n}{n-\alpha}+\delta}} \quad \text{is quasi non-decreasing in $(0, \infty)$.}
\end{equation}
\end{lem}
\begin{proof}
{Part (i).} Set 
\begin{align}
    \label{2024-27}
    t_0 = \bigg(\int _0^\infty\bigg(\frac t{A(t)}\bigg)^{\frac \alpha {n-\alpha}}\, dt\bigg)^{\frac {n-\alpha}{n}}\in (0, \infty].
\end{align}
One has that $A_{\frac{n}{\alpha}}(t)<\infty$ if $t\in [0, t_0)$. Moreover, if $t_0<\infty$, namely if \eqref{convinf} is in force, then 
$A_{\frac{n}{\alpha}}(t)=\infty$ for $t\in (t_0, \infty)$. Thus, it suffices to prove that, whatever $t_0$ is, 
\begin{align}
    \label{2024-261}
    \dfrac{A_{\frac{n}{\alpha}}(t)}{t^{\frac{n}{n-\alpha}+\delta}} \quad \text{is non-decreasing in $(0, t_0)$.}
\end{align}
 Since $A \in \n$, by the property \eqref{equivnabla2} there exists $\varepsilon >0$ such that 
\begin{equation}\label{nov21}
\dfrac{ta(t)}{A(t)}\geq 1+\varepsilon \quad \text{for $t\in (0,t_0)$,}
\end{equation}
where $a$ is the function appearing in the representation formula \eqref{young}.
Given $\delta >0$, one has that 
\begin{align}\label{nov22}
\frac{d}{dt} \bigg(\frac{A_{\frac{n}{\alpha}}(t)}{t^{\frac{n}{n-\alpha}+\delta}}\bigg) >0  \quad \text{for $t\in (0,t_0)$}
\end{align}
if and only if
\begin{align}\label{nov23}
\frac{tA_{\frac{n}{\alpha}}'(t)}{A_{\frac{n}{\alpha}}(t)}>  \frac{n}{n-\alpha}+\delta \quad \text{for $t\in (0,t_0)$.}
\end{align}
On the other hand,
\begin{align}\label{nov24}
\frac{tA_{\frac{n}{\alpha}}'(t)}{A_{\frac{n}{\alpha}}(t)} 
= \frac{t a(H_{A,\frac{n}{\alpha}}^{-1}(t))}{A(H^{-1}_{A,\frac{n}{\alpha}}(t)) H'_{A,\frac{n}{\alpha}}(H^{-1}_{A,\frac{n}{\alpha}}(t))} \quad \text{for $t\in (0,t_0)$.}
\end{align}
We have that
\begin{align}\label{nov25}
\frac{H_{A,\frac{n}{\alpha}}(s) a(s)}{A(s) H'_{A,\frac{n}{\alpha}}(s)} &=  \frac {n}{n-\alpha}s^{-\frac{\alpha}{n-\alpha}} a(s) A(s)^{\frac \alpha{n-\alpha}-1}\int_0^s \bigg(\frac r{A(r)}\bigg)^{\frac \alpha{n-\alpha}}dr 
\\  \nonumber & \geq 
\frac {n}{n-\alpha}(1+\varepsilon)   s^{-\frac{n}{n-\alpha}}  A(s)^{\frac \alpha{n-\alpha}}\int_0^s \bigg(\frac r{A(r)}\bigg)^{\frac \alpha{n-\alpha}}dr
\\  \nonumber & \geq 
\frac {n}{n-\alpha}(1+\varepsilon)   
\quad \text{for $s>0$,}
\end{align}
where the first inequality holds thanks to equation \eqref{nov21}, and the second one since, by \eqref{2024-5},
$$\frac 1s\int_0^s \bigg(\frac r{A(r)}\bigg)^{\frac \alpha{n-\alpha}}dr \geq  \bigg(\frac s{A(s)}\bigg)^{\frac \alpha{n-\alpha}} \quad \text{for $s>0$.}
$$
 Owing to equation \eqref{nov24}, an application of the inequality \eqref{nov25} with $s=H^{-1}_{A,\frac{n}{\alpha}}(t)$ yields \eqref{nov23}, whence \eqref{nov22} and \eqref{2024-261} follow.
 \\
 {Part (ii).} As noticed above, under the assumption \eqref{convinf}, one has that $t_0<\infty$. Since we are assuming that $A\in \n$   near $0$, there exist constants $\varepsilon$ and $t_1$
  such that
  \begin{align}
      \label{2024-265}
      \dfrac{ta(t)}{A(t)}\geq 1+\varepsilon \quad \text{for $t\in (0,t_1)$.}
  \end{align}
Set $t_2=\min\{t_0, t_1\}$.
 The argument offered in the proof of Part (i) tells us that 
\begin{equation*}
\dfrac{A_{\frac{n}{\alpha}}(t)}{t^{\frac{n}{n-\alpha}+\delta}} \quad \text{is increasing in $(0, t_2)$.}
 \end{equation*}
 If $t_2=t_0$, then \eqref{nov20}, and hence \eqref{nov20bis}, hold, inasmuch as $A_{\frac{n}{\alpha}}(t)=\infty$ for $t>t_0$. On the other hand, if $t_2=t_1$, then, by the monotonicity of $A_{\frac{n}{\alpha}}$,
 $$\dfrac{A_{\frac{n}{\alpha}}(t)}{t^{\frac{n}{n-\alpha}+\delta}}\leq \Big(\frac{t_0}{t_1}\Big)^{\frac{n}{n-\alpha}+\delta}
 \dfrac{A_{\frac{n}{\alpha}}(t_0)}{t_0^{\frac{n}{n-\alpha}+\delta}} 
 \qquad \text{if $t\in (t_1,t_0)$.}$$
 Altogether, the property \eqref{nov20bis} follows also in this case.
\end{proof}

Let $\widehat  H_{A,\frac{n}{\alpha}}   : [0, \infty) \to [0, \infty)$ be a variant of the function $H_{A,\frac{n}{\alpha}}$, defined with $\frac{t}{A(t)}$ replaced with $\frac{1}{a(t)}$, where $a$ is the function appearing in \eqref{young}.
Namely, 
\begin{equation}\label{hatH}
\widehat   H_{A,\frac{n}{\alpha}}(t)= \bigg(\int _0^t \bigg(\frac 1{a(\tau)}\bigg)^{\frac {\alpha}{n-\alpha}}\, d\tau\bigg)^{\frac {n-\alpha}{n}} \quad \hbox{for $t \geq 0$.}
\end{equation}
Moreover,
 let $ \widehat  A_{\frac{n}{\alpha}}$ be the Young function given by
\begin{equation}\label{hatAn}
  \widehat A_{\frac{n}{\alpha}}(t)= A\big(  \widehat H_{A,\frac{n}{\alpha}}^{-1}(t)\big) \quad \hbox{for $t \geq 0$.}
\end{equation}
The latter function is equivalent to the original Sobolev conjugate $A_{\frac{n}{\alpha}}$. This is the content of the following lemma. 

\begin{lem}\label{Anbis}
Let $A$ be a  finite-valued Young function fulfilling the condition \eqref{conv-alpha} and let $A_{\frac{n}{\alpha}}$ and $\widehat A_{\frac{n}{\alpha}}$ be the functions defined by  \eqref{A-alpha} and \eqref{hatAn}.
Then,  
\begin{equation}\label{equivsob}
 \widehat  A_{\frac{n}{\alpha}}(t/2)\leq     {A_{\frac{n}{\alpha}}}(t) \leq    \widehat A_{\frac{n}{\alpha}}(t) \quad  \hbox{for $t \geq 0$.}
\end{equation}
\end{lem}
\begin{proof}
From  the inequalities \eqref{aA} one can deduce that
\begin{equation}
    \label{2024-53}
    \widehat   H_{A,\frac{n}{\alpha}}(t)\leq  H_{A,\frac{n}{\alpha}}(t)\leq 2\widehat   H_{A,\frac{n}{\alpha}}(t) \quad \text{for $t \geq 0$.}
\end{equation}
Equation \eqref{equivsob} follows from \eqref{2024-53}.
\end{proof}

We still need to introduce more
 functions  associated with 
 the Young function $A$. They are the function
 $\overline H_{A,\frac{n}{\alpha}} : (0, \infty) \to [0, \infty)$  given by
\begin{equation}\label{overA}
\overline{H}_{A,\frac{n}{\alpha}}(t)=\frac{t}{\widehat H_{A,\frac{n}{\alpha}}(t)} \qquad \text{for $t>0$,}
\end{equation}
and the Young function $B_{A,\frac{n}{\alpha}}$  defined as
\begin{equation}\label{E:B}
B_{A,\frac{n}{\alpha}}(t)=\int_0^t \frac{A(\overline{H}_{A,\frac{n}{\alpha}}^{-1}(s))}{s}\,ds \quad \text{for} \quad t>0.
\end{equation}
\\
Moreover, let $E$ be the function from \eqref{Enew}, with $\ell =\alpha$. Namely,
\begin{equation}\label{Ealpha}
E(t) = \int_0^{A^{-1}(t)} a(s)^{\frac{n}{\alpha}}\, ds  \qquad \text{for $t \geq 0$.}
\end{equation}
Then, we denote by
 $B_{E,\frac{n}{\alpha}}$ the function obtained from $E$, via a process analogous to the one which produces $B_{A,\frac{n}{\alpha}}$ from $A$. In particular, in \eqref{hatH}, the function $a(t)$ has to be replaced with $(a(A^{-1}(t)))^{\frac{n}{\alpha}-1}$, the left-continuous derivative of $E$. Namely,  
\begin{align}\label{2024-30}
  \widehat   H_{E,\frac{n}{\alpha}}(t)= \bigg(\int _0^t  \frac 1{a(A^{-1}(\tau))}\, d\tau\bigg)^{\frac {n-\alpha}{n}} \quad \hbox{for $t \geq 0$,}  
\end{align}
\begin{equation}\label{2024-31}
\overline{H}_{E,\frac{n}{\alpha}}(t)=\frac{t}{\widehat H_{E,\frac{n}{\alpha}}(t)} \qquad \text{for $t>0$,}
\end{equation}
 and 
\begin{equation}\label{2024-32}
B_{E,\frac{n}{\alpha}}(t)=\int_0^t \frac{E(\overline{H}_{E,\frac{n}{\alpha}}^{-1}(s))}{s}\,ds \quad \text{for} \quad t\geq 0.
\end{equation}
Also, we set 
\begin{equation}\label{hatEn}
  \widehat E_{\frac{n}{\alpha}}(t)= E\big(  \widehat H_{E,\frac{n}{\alpha}}^{-1}(t)\big) \quad \hbox{for $t \geq 0$.}
\end{equation}

\begin{lem}\label{lem:L:B}
Let $A$ be a   finite-valued Young function satisfying~\eqref{conv-alpha} and let  $\widehat H_{A,\frac{n}{\alpha}}$ be the function defined by equation \eqref{hatH}. Then, $\overline{H}_{A,\frac{n}{\alpha}}$
is non-decreasing,  $B_{A,\frac{n}{\alpha}}$ is a Young function and
\begin{equation}\label{E:equivalence_B}
A(\overline{H}_{A,\frac{n}{\alpha}}^{-1}(t/2)) \leq B_{A,\frac{n}{\alpha}}(t) \leq A (\overline{H}_{A,\frac{n}{\alpha}}^{-1}(t)) \quad \text{for} \quad t>0.
\end{equation}
\end{lem}


\begin{proof}
We start by showing that $\overline{H}_{A,\frac{n}{\alpha}}$ is non-decreasing. To this end, it suffices to show that the function
$$
t\mapsto \frac{t^{\frac{n}{n-\alpha}}}{\displaystyle\int_0^t \bigg(\frac{1}{a(s)}\bigg)^{\frac{\alpha}{n-\alpha}}\,ds}
$$
is non-decreasing on $(0,\infty)$. This is  true, since the derivative of this function times  $\big(\int_0^t 1/(a(s))^{\alpha/(n-\alpha)}\,ds\big)^2$ equals
$$
t^{\frac{n}{n-\alpha}}\left(\frac{n}{(n-\alpha)\,t} \int_0^t \left(\frac{1}{a(s)}\right)^{\frac{\alpha}{n-\alpha}}\,ds - \left(\frac{1}{a(t)}\right)^{\frac{\alpha}{n-\alpha}}\right), 
$$
and the last expression is nonnegative thanks to the facts that $\frac{n}{n-\alpha}\geq 1$ and $a(t)$ is non-decreasing. 

To show that $B_{A,\frac{n}{\alpha}}$ is a Young function we need to verify that $A(\overline{H}_{A,\frac{n}{\alpha}}^{-1}(s))/s$ is non-decreasing on $(0,\infty)$. Since $\overline{H}_{A,\frac{n}{\alpha}}$ is non-decreasing, it suffices to prove that $A(t)/\overline{H}_{A,\frac{n}{\alpha}}(t)$ is non-decreasing. The latter function can be rewritten as $(A(t)/t) \widehat H_{A,\frac{n}{\alpha}}(t)$, which is a non-decreasing function, inasmuch as   both $A(t)/t$ and $\widehat H_{A,\frac{n}{\alpha}}(t)$ enjoy this property.
\\
Finally, equation \eqref{E:equivalence_B} is an easy consequence of the increasing monotonicity of the function $A(\overline{H}_{A,\frac{n}{\alpha}}^{-1}(t))/t$. 
\end{proof}

\begin{lem}\label{prop1}
 Let $A$ be a   finite-valued Young function satisfying the condition \eqref{A}.  Then, the function $E$ defined by \eqref{Ealpha} is a Young function such that $E \in \Delta_2 \cap \nabla_2$. 
\end{lem}
\begin{proof}
A change of variables in the integral on the right-hand side of equation \eqref{Ealpha} yields the alternate formula:
\begin{equation}\label{Ee}
E(t)=\int_0^t (a(A^{-1}(r)))^{\frac{n}{\alpha}-1}\,dr \qquad \text{for $t \geq 0$.}
\end{equation}
Since both $a$ and $A^{-1}$ are non-decreasing functions, equation \eqref{Ee} ensures that $E$ is a Young function. Furthermore, 
Lemma~\ref{L:lemma_young} applied with $K=2$ tells us that $E\in \Delta_2$, and the same lemma applied with
 $K=2^{\frac{n}{\alpha}-1}$
 tells us that $E\in \nabla_2$. 
\end{proof}

\begin{lem}\label{Enequiv}
Let $A$ be a   finite-valued Young function and let $E$ be the function  defined by \eqref{Ealpha}. Then, 
\begin{equation}\label{conv-Ealpha}
\int _0\bigg(\frac t{E (t)}\bigg)^{\frac {\alpha}{n-\alpha}}\, dt <
\infty,
\end{equation}
 and 
\begin{equation}\label{divinf-alpha}
\int^\infty \left(\frac{t}{E(t)}\right)^{\frac{\alpha}{n-\alpha}}\,dt=\infty.
\end{equation}
Moreover, the Sobolev conjugate $E_{\frac{n}{\alpha}}$ of $E$, defined as in \eqref{A-alpha}, fulfills:
\begin{equation}\label{nov11}
E_{\frac{n}{\alpha}}(t) \simeq  \int_0^{t^{\frac{n}{n-\alpha}}} a(s)^{\frac{n}{\alpha}}\, ds  \qquad \text{for $t \geq 0$.}
\end{equation}
\end{lem}
\begin{proof}
Equation~\eqref{Ee} yields the formula
\begin{align}
    \label{2024-269}
    E'(t)= a(A^{-1}(t))^{\frac{n}{\alpha}-1}  \quad \hbox{for $t \geq 0$,}
\end{align}
for the left-continuous derivative of $E$.
Therefore,  
\begin{align}\label{nov12}
\int _0^t \bigg(\frac 1{E'(\tau)}\bigg)^{\frac {\alpha}{n-\alpha}}\, d\tau = \int _0^t \frac{d\tau}{ a(A^{-1}(\tau))} = \int _0^{A^{-1}(t) }d\tau= A^{-1}(t)  \quad \hbox{for $t \geq 0$.}
\end{align}
The properties \eqref{conv-Ealpha} and \eqref{divinf-alpha}    follow from \eqref{nov12}, via equation \eqref{aA} applied to $E$. 
\\ Finally, equation \eqref{nov11} is a consequence of   \eqref{2024-269} and of Lemma \ref{Anbis}, applied with $A$ replaced with $E$.
\end{proof}

\begin{lem}
    \label{EF}
    Assume that $A$ is a  finite-valued  Young function 
    and $E$ is defined as in \eqref{Ealpha}. Let $B_{E,\frac n\alpha}$ be the function associated with $E$ by \eqref{2024-31}.  Assume that $F$ is a Young function satisfying \eqref{F} with $\ell = \alpha$. 
    If the inequality \eqref{HNew} holds with $\ell = \alpha$, then there exists a constant $c$ such that
    \begin{align}
        \label{2024-39}
    B_{E,\frac n\alpha}^{-1}(t) F_{\frac{n}{\alpha}}^{-1}(t)  \leq c F^{-1}(t) \quad \text{for $t>0$.}
    \end{align}
\end{lem}
\begin{proof}
    An application of Lemma \ref{lem:L:B}, with $A$ replaced with $E$, tells us that
    \begin{align}\label{2024-40}
    B_{E,\frac n\alpha}^{-1}(t)\leq 2
        \overline H_{E,\frac n \alpha}\big(E^{-1}(t)\big)= 2 \frac{E^{-1}(t)}{\widehat H_{E,\frac n \alpha}(E^{-1}(t))}= 2 \frac{E^{-1}(t)}{\widehat E_{\frac{n}{\alpha}}^{-1}(t)}
        \quad \text{for $t>0$.}
    \end{align}
    From Lemma \ref{Anbis}, applied with $A$ replaced with $E$, we infer that
    $$\widehat E_{\frac{n}{\alpha}}^{-1}(t) \geq c' E_{\frac{n}{\alpha}}^{-1} (t)\quad \text{for $t>0$,} $$
    for some positive constant $c'$.
    Therefore, equation  \eqref{2024-40} implies that
    \begin{align}\label{2024-41}
    B_{E,\frac n\alpha}^{-1}(t)\leq 
    \frac 2{c'}\frac{E^{-1}(t)}{E_{\frac{n}{\alpha}}^{-1}(t)}
        \quad \text{for $t>0$.}
    \end{align}
    Coupling equation \eqref{HNew} with \eqref{2024-41} yields the  inequality  \eqref{2024-39}.
\end{proof}

\begin{lem}
    \label{lem:MnD}
    Assume that $A$ is a   finite-valued Young function. 
    Let $D$ be  defined as in \eqref{D}, let $E$ be  defined as in \eqref{E}, let $\overline H_{E,n}$ be defined as in \eqref{2024-31} with $\alpha=1$, and let $M_n$ be the Sobolev conjugate of a function $M$ obeying \eqref{M}. Then, 
    \begin{align}
        \label{2024-51}
      M_n(D^{-1}(t)) \leq E(\overline H_{E,n}^{-1}(t))  \quad \text{near $0$,}
    \end{align}
and 
\begin{align}
        \label{2024-51bis}
 M_n(2D^{-1}(t)) 
       \geq   
       E(\overline H_{E,n}^{-1}(t))
        \quad \text{near $0$.}
    \end{align}
%
  %
\end{lem}
\begin{proof}
    From equation \eqref{2024-53} and the property \eqref{aA}, both applied with $A$ replaced with $E$, 
    one infers that
\begin{align}\label{2024-54}
       H_{E,n}^{-1}\big(  \widehat H_{E,n}(t)\big)\leq t \quad \text{for $t \geq 0$,}
    \end{align}
    and 
\begin{align}\label{2024-54bis}
       H_{E,n}^{-1}\big( 2 \widehat H_{E,n}(t)\big)\geq t \quad \text{for $t \geq 0$.}
    \end{align}
    Since $M(t)=E(t)$ near $0$, 
    \begin{align}
        \label{2024-55}
        M_n(D^{-1}(t))= E_n(D^{-1}(t)) = E\big( H_{E,n}^{-1} (D^{-1}(t))\big) \quad \text{near $0$.}
    \end{align}
Equation ~\eqref{nov12} tells us  that  $\widehat H_{E, n}(t)=A^{-1}(t)^{\frac{n-1}{n}}$ for $t>0$. Therefore,
\begin{align}\label{2024-50}
D(t)=\frac{\widehat H_{E, n}^{-1}(t)}{t}
=\overline{H}_{E, n}(\widehat H^{-1}_{E, n}(t)) \qquad \text{for $t>0$.}
\end{align}
Combining equations \eqref{2024-55} and \eqref{2024-50} yields
\begin{align}
    \label{2024-56}
       M_n(D^{-1}(t))=  E\big( H_{E,n}^{-1} (\widehat H_{E, n}(\overline{H}^{-1}_{E, n}(t)))) \quad \text{near $0$.}
\end{align}
Thanks to \eqref{2024-54}  we hence deduce that
\begin{align}
    \label{2024-57}
    M_n(D^{-1}(t)) 
       \leq E(\overline H_{E,n}^{-1}(t))
       \quad \text{near $0$,}
\end{align}
namely \eqref{2024-51}. Similarly, 
from \eqref{2024-54bis} one obtains that
\begin{align}
    \label{2024-57bis}
    M_n(2D^{-1}(t)) 
       \geq  E\big( H_{E,n}^{-1} (2\widehat H_{E, n}(\overline{H}^{-1}_{E, n}(t))))
       \geq
       E(\overline H_{E,n}^{-1}(t))
       \quad \text{near $0$,}
\end{align}
namely \eqref{2024-51bis}.
\end{proof}

\section{Boundedness properties of integral operators in Orlicz spaces}\label{S:4}


In this section, we establish boundedness properties in Orlicz spaces of the integral operators involved in representation formulas for the solutions to the problems \eqref{eq:main-eq} and \eqref{eq:gen-eq}.
%
Some specific properties of Orlicz norms built upon the Young functions introduced in the previous sections are also collected in the final part of the section.

Let $\alpha \in (0,n)$.  We denote by
$T_{\alpha}$ the  extension operator  defined by 
$$T_{\alpha}h(x,x_{n+1})=\pi^{\frac{n+1}{2}}\frac{\Gamma(\frac{n-\alpha}{2})}{\Gamma(\frac{\alpha+1}{2})}\int_{\rn}\frac{h(y)}{(|x-y|^{2}+x_{n+1}^{2})^{\frac{n-\alpha}{2}}}dy \quad \text{for a.e. $(x,x_{n+1})\in \R^{n+1}$,}$$
for
 $h\in L^1\big(\rn,(1+|x|)^{-(n-\alpha)}dx\big)$. 
 \\ Furthermore, let $\mathcal H$ be the Poisson extension operator given by
$$\mathcal{H}f(x,x_{n+1})=\pi^{-\frac{n}{2}}\Gamma(n/2)\int_{\rn}\frac{x_{n+1}f(y)}{(|x-y|^{2}+x_{n+1}^{2})^{\frac{n+1}{2}}}dy\quad \text{for $(x,x_{n+1})\in \R^{n+1}$,}$$
for
 $f\in L^1\big(\rn,(1+|x|)^{-(n+1)}dx\big)$. 
\\
 Also, recall that the classical Riesz potential $I_\alpha$ is given by
$$
I_\alpha f (x)=\frac{\Gamma(\frac{n-\alpha}{2})}{2^{\alpha}\pi^{n/2}\Gamma(\alpha/2)}  \int_{\rn} \frac {f(y)}{|x-y|^{n-\alpha}}\, dy, 
\qquad  \text{for  a.e. $x\in \rn$,}
$$
for $f \in  L^1\big(\rn,(1+|x|)^{-(n-\alpha)}dx\big)$, and the Riesz transform
$\mathcal{R}_{j}$ is defined, for $j=1,\dots,n$,  via the principal value as a singular integral operator, by
\begin{align*}
\mathcal{R}_jf(x)= \dfrac{2}{\omega_n}\int_{\mathbb{R}^n}f(y)\dfrac{x_j-y_j}{|x-y|^{n+1}}dy \qquad \text{for   a.e. $x\in \rn$,}
\end{align*} 
for   $f \in  L^1\big(\rn,(1+|x|)^{-n}dx\big)$.
The following properties of the operator $T_{\alpha}$ are proved in \cite[Lemma 2.6]{Di}. In what follows, $\mathcal S(\rn)$ stands for the Schwartz space of those smooth functions in $\rn$ which, together will their derivatives of any order,  decay near infinity faster than any power.

\begin{lem}\label{lem:kernel-identities}
Let $\ell$  be an odd integer such that $0<\ell <n$. Then, there exist constants $c_1$ and $c_2$, depending on $n$ and $\ell$, such that 
\begin{equation}\label{Ker-id}
\begin{cases}(-1)^{\frac{\ell+3}{2}}\partial_{x_{n+1}}\Delta^{\frac{\ell-1}{2}}T_{\ell}h=c_1\mathcal{H}h
\\ 
 (-1)^{\frac{\ell+1}{2}}\partial_j\Delta^{\frac{\ell-1}{2}}T_{\ell}h=c_2\mathcal{R}_j\mathcal{H}h=c_2\mathcal{H}\mathcal{R}_jh\hspace{0.2cm}\mbox{for}\hspace{0.2cm} j=1,2,...,n,
\end{cases}
\end{equation}	
for $h\in \mathcal{S}(\rn)$.
\end{lem}


Thanks to this lemma, boundedness properties of any $\ell$-th order derivative of $T_{\ell}$ in Orlicz spaces  can be deduced from parallel properties of the Riesz transform and the Poisson extension operator. This is accomplished in the next lemma.

\begin{lem}\label{lem:T-bounds} Assume that $\ell \in (0,n)$ is an odd integer. Let $A$ be a   Young function fulfilling the condition ~\eqref{conv-alpha}, with $\alpha=\ell$, and
such  that $A\in \Delta_2\cap \nabla_2$.
Then, 
$$T_{\ell}: L^{A}(\rn) \to V^{\ell,(A_\diamond,A_{{n}/{\ell}})}(\rnp,\rn).$$ Namely, there exists a constant $c$ such that
\begin{equation}\label{bound-on-T}
\|\nabla^{\ell}T_{\ell}h\|_{L^{A_\diamond}(\rnp)}+\|T_{\ell}h\|_{L^{A_{{n}/{\ell}}}(\rn)}\leq c\|h\|_{L^{A}(\rn)}    
\end{equation}
for 
 $h\in L^{A}(\rn)$. 
\end{lem}

The following Sobolev type boundedness property of the Riesz potential comes into play in the proof of Lemma \ref{lem:T-bounds}.

\begin{lem}
    \label{Sobolev-Riesz}
    Let $\alpha \in (0,n)$ and let $A$ be a   finite-valued Young function fulfilling the condition \eqref{conv-alpha}. Assume that either
 $A\in \nabla _2$ globally, or $A\in \nabla _2$ near $0$ and the condition \eqref{convinf} holds. Then,
\begin{equation}\label{nov301}
I_{\alpha} : L^{A}(\rn) \to L^{A_{\frac{n}{\alpha}}}(\rn).
\end{equation}
\end{lem}
	\begin{proof}
	    A proof of equation \eqref{nov301} makes use of  \cite[Theorem 7.2.1, item (ii)]{M}, whose proof is, in turn, related to a general characterization of boundeness properties of Riesz potentials in Orlicz spaces from \cite{A1}.  That theorem enables one to deduce equation \eqref{nov301} after verifying that:
\begin{equation}\label{nov30}
\inf _{0<t<1}A(t)t^{-\frac{n}{\alpha}}>0
\end{equation}
and 
\begin{equation}\label{nov31}
\int_{0}^{t}\dfrac{G(s)}{s^{1+\frac{n}{n-\alpha}}}ds\leq \dfrac{P^{-1}(ct)}{t^{\frac{n}{n-\alpha}}}\hspace{0.2cm}\text{for}\hspace{0.2cm}t\geq 0,
\end{equation} 
for some constant $c>0$, where $P:(0, \infty ) \to [0, \infty)$ is  the function given by 
$$P(t)= \sup_{0<s\leq t}A^{-1}(s)s^{-\frac{\alpha}{n}} \quad \text{for $t>0$,}$$
and $G:(0, \infty ) \to [0, \infty)$ is the function defined as 
$$G(t) = \bigg(\frac {tA(H_{A,\frac{n}{\alpha}}^{-1}(t))}{H_{A,\frac{n}{\alpha}}^{-1}(t)}\bigg)^{ \frac{n}{n-\alpha}}\quad \text{for $t>0$.}$$
Equation \eqref{nov30} is a consequence of the fact that 
$$\lim _{t\to 0} A(t)t^{-\frac{n}{\alpha}}=\infty,$$
 which  follows from the assumption \eqref{conv-alpha} and the property \eqref{2024-5}. 
\\ As far as the condition \eqref{nov31} is concerned, we observe that, owing to \cite[Lemma 1]{Cia_CPDE}, one has that $G \simeq A_{\frac{n}{\alpha}}$. Hence, the inequality \eqref{nov31} is equivalent to 
\begin{equation}\label{eq:condition-Riesz}\int_{0}^{t}\dfrac{A_{\frac{n}{\alpha}}(s)}{s^{1+\frac{n}{n-\alpha}}}ds\leq \dfrac{P^{-1}(ct)}{t^{\frac{n}{n-\alpha}}}\hspace{0.2cm}\text{for}\hspace{0.2cm}t\geq 0,
\end{equation} 
for some constant $c>0$. In order to prove   \eqref{eq:condition-Riesz},  note that, 
by Lemma \ref{quotient}, there exists $\delta >0$ such that
the function $\dfrac{A_{\frac{n}{\alpha}}(t)}{t^{\frac{n}{n-\alpha}+\delta}}$ is quasi non-decreasing. Thus, there exists a constant $c'$ such that
\begin{align*}
\int_{0}^{t}\dfrac{A_{\frac{n}{\alpha}}(s)}{s^{\frac{n}{n-\alpha}+1}}ds=\int_{0}^{t}\dfrac{A_{\frac{n}{\alpha}}(s)}{s^{\frac{n}{n-\alpha}+\delta}}s^{\delta-1}ds
\leq c'\dfrac{A_{\frac{n}{\alpha}}(t)}{t^{\frac{n}{n-\alpha}+\delta}} \int_{0}^{t}s^{\delta-1}ds
= \frac{c'}{\delta}\dfrac{A_{\frac{n}{\alpha}}(t)}{t^{\frac{n}{n-\alpha}}} \qquad\text{for $t \geq 0$.}
\end{align*}  
Thanks to the latter chain, equation \eqref{eq:condition-Riesz} will follow if we show that there exists a positive constant $c$ such that $A_{\frac{n}{\alpha}}(t)\leq P^{-1}(ct)$, namely, $A(H_{A,\frac{n}{\alpha}}^{-1}(t))\leq P^{-1}(ct)$ or,  equivalently, 
\begin{align}\label{2024-15}
    c\,H_{A,\frac{n}{\alpha}}(A^{-1}(t))\geq P(t) \qquad \text{for $t\geq 0$.}
\end{align}
Since $H_{A,\frac{n}{\alpha}}(A^{-1}(t))$ is an increasing function, equation \eqref{2024-15} will 
be established 
 if we prove that
\begin{align}\label{2024-16}
     c H_{A,\frac{n}{\alpha}}(A^{-1}(t))\geq A^{-1}(t)t^{-\frac{\alpha}{n}} \qquad \text{for $t\geq 0$.}
\end{align}
This inequality holds as a consequence of the chain:
\begin{align*}
H_{A,\frac{n}{\alpha}}(t)&=\bigg(\int_{0}^{t}\bigg(\dfrac{s}{A(s)}\bigg)^{\frac{\alpha}{n-\alpha}}ds\bigg)^{\frac{n-\alpha}{n}}
\geq \frac{1}{A(t)^{\frac{\alpha}{n}}}\bigg(\int_{0}^{t}s^{\frac{\alpha}{n-\alpha}}ds\bigg)^{\frac{n-\alpha}{n}}=\bigg(\dfrac{n-\alpha}{n}\bigg)^{\frac{n-\alpha}{n}}\dfrac{t}{A(t)^{\frac{\alpha}{n}}} \qquad \text{for $t\geq 0$.}   
\end{align*} 
\end{proof}
 The  following lemma ensures that the operator 
$T_\alpha$ is well-defined for functions in the Orlicz space $L^A(\rn)$, provided that $A$ fulfils the condition \eqref{conv-alpha}.

\begin{lem}
    \label{inclusion}
    Let $\alpha \in (0,n)$ and let $A$ be Young function fulfilling the condition \eqref{conv-alpha}. Then,
    \begin{align}
        \label{inclusion1}
        L^A(\rn) \to  L^1\big(\rn,(1+|x|)^{-(n-\alpha)}dx\big).
    \end{align}
\end{lem}
\begin{proof}
    By the H\"older inequality \eqref{holder},
    \begin{align}
        \label{2024-110}
\int_{\rn}\frac{|h(x)|}{(1+|x|)^{(n-\alpha)}}\, dx \leq 2 \|h\|_{L^A(\rn)} \big\|(1+|x|)^{-(n-\alpha)}\|_{L^{\widetilde A}(\rn)}
    \end{align}
    for $h \in L^A(\rn)$. Via the very definition of Luxemburg norm, one can verify that
    $$\big\|(1+|x|)^{-(n-\alpha)}\|_{L^{\widetilde A}(\rn)} <\infty\quad \text{if and only if} \quad\int_0\frac{\widetilde A(t)}{t^{1+\frac {n}{n-\alpha}}}\, dt < \infty.$$
    Equation \eqref{inclusion1} thus follows, since the convergence of the last integral is equivalent to the condition \eqref{conv-alpha} -- see \cite[Lemma 2.3]{cianchi-ibero}.
\end{proof}

\begin{proof}[Proof of Lemma~\ref{lem:T-bounds}]
  To begin with, observe that, given a function $h : \rn \to \R$,  the restriction 
  of  $T_{\ell}h$ to $\rn$ agrees with  a dimensional  multiple of the Riesz potential $I_{\ell}h$ on $\rn$. Thus,
owing to equation \eqref{nov301}, there exists a constant $c$
such that
\begin{align}\label{2024-21}
\|T_{\ell}h\|_{L^{A_{\frac{n}{\ell}}}(\rn)}\leq c\|h\|_{L^{A}(\rn)}    
\end{align}
for $h \in L^{A}(\rn)$.
\\ A parallel bound for the first norm on the left-hand side of the inequality \eqref{bound-on-T} relies upon 
 an interpolation argument.  Recall that the Poisson extension operator has the following boundedness properties:
\begin{align}\label{2024-22}
\mathcal{H}:L^{1}(\rn)\rightarrow L^{\frac{n+1}{n},\infty}(\rnp)\quad \text{and}\quad \mathcal{H}:L^{\infty}(\rn)\rightarrow L^{\infty}(\rnp),
\end{align} 
see e.g. \cite{HWY}. 
From the Theorem of \cite{A2}, one can hence deduce that
\begin{align}
    \label{2024-17}
    \mathcal{H}:L^{A}(\rn)\rightarrow L^{A_\diamond}(\rnp),
\end{align}
provided that there exists a constant $c$ such that 
$$\int_{0}^{t}\frac{A_\diamond(s)}{s^{\frac{n+1}{n}+1}}ds \leq c\bigg(\frac{A(t)}{t}\bigg)^{\frac{n+1}{n}}\hspace{0.1cm}\text{for}\hspace{0.1cm}t>0.$$
 The latter inequality reads
\begin{align}\label{2024-18}
\int_{0}^{t}\frac{A(s)^{\frac{n+1}{n}}}{s^{\frac{n+1}{n}+1}}ds\leq c\bigg(\frac{A(t)}{t}\bigg)^{\frac{n+1}{n}}\hspace{0.1cm}\text{for}\hspace{0.1cm}t>0.
\end{align}
Since $A\in \n$, the property 
\eqref{equivnabla2}
ensures that the function $\dfrac{A(t)}{t^{1+\varepsilon}}$ is  non-decreasing  for some $\varepsilon>0$. Therefore, equation \eqref{2024-18} follows from the following chain:
$$\int_{0}^{t}\frac{A(s)^{\frac{n+1}{n}}}{s^{\frac{n+1}{n}+1}}ds=\int_{0}^{t}\frac{A(s)^{\frac{n+1}{n}}}{s^{\frac{n+1}{n}+\varepsilon}}s^{\varepsilon-1}ds\leq \frac{A(t)^{\frac{n+1}{n}}}{t^{\frac{n+1}{n}+\varepsilon}}\int_{0}^{t}s^{\varepsilon-1}ds= \frac 1\varepsilon \bigg(\frac{A(t)}{t}\bigg)^{\frac{n+1}{n}}  \text{for $t>0$.}$$
The assumption   $A\in \Delta_2\cap \nabla_2$ classically ensures that
 \begin{align}\label{2024-19}
     \mathcal{R}_j:L^{A}(\rn)\rightarrow L^{A}(\rn),
 \end{align} 
 see e.g. \cite{KoKr}.
 From Lemma \ref{lem:kernel-identities} and equations \eqref{2024-17} and \eqref{2024-19} one deduces that
\begin{align}\label{2024-20}
\|\nabla^{\ell} T_{\ell}h\|_{L^{A_\diamond}(\rnp)}&\leq  c\bigg(\sum_{j=1}^n \|\mathcal{H} \mathcal{R}_jh\|_{L^{A_\diamond}(\rnp)}+\| \mathcal{H}h\|_{L^{A_\diamond}(\rnp)}\bigg)\leq c'\|h\|_{L^{A}(\rn)}
\end{align} 
for some constants   $c$ and $c'$, and for every $h \in \mathcal S(\rn)\cap L^{A}(\rn)$.
\\   It remains to show that the same inequality continues to hold  for any   $h\in L^A(\rn)$. To this end,  observe that,
given $L>0$, the kernel $K_\ell$ associated with the operator $T_\ell$ satisfies the bounds
\begin{align}
    \label{2024-120}|\nabla ^m_x K_{\ell}(x-y,x_{n+1})| \leq \frac {c} {(|x-y|^{2}+L^{2})^{\frac{n-\ell}{2}}} \quad \text{for $x, y \in \rn$ and $x_{n+1}\geq L$,}
\end{align}
and 
\begin{align}
    \label{2024-120b}|\partial_{x_{n+1}}^m K_{\ell}(x-y,x_{n+1})| \leq \frac {c} {(|x-y|^{2}+L^{2})^{\frac{n-\ell}{2}}} \quad \text{for $x, y \in \rn$ and $x_{n+1}\geq L$,}
\end{align}
for $m \in \N \cup \{0\}$, and some constant $c=c(n,\ell, m, L)$. 
Hence, by Lemma \ref{inclusion}, if $h \in L^A(\rn)$, then, fixing $x_{n+1}>0$,
\begin{align}
    \label{2024-121} 
    h(y)\nabla ^m K_{\ell}(x-y,x_{n+1})\in L^1(B_R(0)\times \rn)
\end{align}
for $m \in \N \cup \{0\}$ and $R>0$, where $B_R(0)$ denotes the ball in $\rn$ centered at $0$, with radius $R$, and $(x,y) \in B_R(0)\times \rn$.  From this piece of information and Fubini's theorem, one deduces that, for every $x_{n+1}>0$, the function
$T_\ell h(\cdot, x_{n+1})$ admits weak derivatives of any order $m$, and 
\begin{align}\label{2024-135}
\nabla ^m T_\ell h (x, x_{n+1}) = \int_{\rn} \nabla ^m K_{\ell}(x-y,x_{n+1})h(y)\,dy \quad \text{for a.e. $x \in \rn$.}
\end{align}
Now, let $\{h_k\}$ be a sequence in $\mathcal S(\rn)\cap  L^A(\rn)$ such that $h_k\rightarrow h$ in $L^A(\rn)$. Such a sequence exists since $A\in  \Delta_2$. We already know that 
\begin{equation}\label{cont-seq}
   \|\nabla^{\ell} T_{\ell} h_k\|_{L^{A_{\diamond}}(\rnp)}\leq c\|h_k\|_{L^A(\rn)}
\end{equation}
for some constant $c$.
A generalization of a classical result for Lebesgue spaces ensures that there exists a subsequence, still denoted by $\{h_k\}$, and a nonnegative function $\overline h \in L^A(\rn)$, such that
$h_k \to h$ a.e.\ and $|h_k|\leq \overline h$ a.e.\ in $\rn$. As a consequence, thanks to \eqref{2024-120} and the dominated convergence theorem, 
$\nabla ^\ell T_\ell h_k \to \nabla ^\ell T_\ell h$ a.e. in $\rnp$. Since Orlicz norms enjoy the Fatou property,  
$$\|\nabla ^\ell T_\ell h\|_{L^{A_{\diamond}}(\rnp)}\leq \liminf_{n\rightarrow \infty}\|\nabla^{\ell} T_{\ell} h_k\|_{L^{A_{\diamond}}(\rnp)}. $$ 
On the other hand, 
$$\lim_{k\to \infty} \|h_k\|_{L^{A}(\rn)}= \|h\|_{L^{A}(\rn)}.$$
From \eqref{cont-seq}
we hence deduce that
\begin{equation}\label{cont-seq'}
   \|\nabla^{\ell} T_{\ell} h\|_{L^{A_{\diamond}}(\rnp)}\leq C\|h\|_{L^A(\rn)}. 
\end{equation}
The inequality \eqref{bound-on-T}
follows from \eqref{2024-21} and \eqref{cont-seq'}.
\end{proof}	

The regularity properties of the Neumann potential operator $T_1$ stated in the next lemma will be of use in  our proof of Theorem \ref{thm:reg2}.

\begin{lem}\label{lem:reg}
Let $M$ be a Young function as in \eqref{M}, and let $M_n$ be its Sobolev conjugate.
If $f\in W^{1,M}(\rn)$, then  $T_1f\in C^1(\overline{\rnp})\cap W^{1,M_n}(\rn)$. Moreover,   there exists a constant $c$ such that
\begin{equation}\label{est-reg}
\|T_1f\|_{C^1(\overline{\rnp})}+\|T_1f\|_{W^{1,M_n}(\rn)}\leq c\|f\|_{W^{1,M}(\rn)}
\end{equation}
for $f\in W^{1,M}(\rn)$.
\end{lem}
\begin{proof}
The definition of the function $M$ guarantees that the assumptions of Lemma \ref{Sobolev-Riesz} are fulfilled with $A$ replaced with $M$. Also, thanks to \eqref{aug1}, $L^{M_n}(\rn) \to L^\infty(\rn)$. Hence,
\begin{align}\label{2024-23}
W^{1,M}(\rn)\to L^{M_n}(\rn)  \to L^{\infty}(\rn).
\end{align}
  Let $f \in W^{1,M}(\rn)$.
Thanks to the bounds  \eqref{2024-120} and \eqref{2024-120b}, with $\ell=1$, for 
the kernel of the operator $T_1$,
an analogous argument as in the proof of Lemma \ref{lem:T-bounds} tells us that $T_1f$ is weakly differentiable and 
\begin{align}\label{2024-300}
\nabla T_1f(x,x_{n+1})& =\bigg(\nabla_x\int_{\rn}K(y,x_{n+1})f(x-y)dy,\int_{\rn}\partial_{x_{n+1}}K(x-y,x_{n+1})f(y)dy \bigg) 
\\  \nonumber &   = \bigg(\int_{\rn}K(x-y,x_{n+1})\nabla f(y)dy,\int_{\rn}\partial_{x_{n+1}}K(x-y,x_{n+1})f(y)dy \bigg)
\end{align}
for a.e. $(x, x_{n+1})\in \rnp$.
Hence, the following chain holds:
\begin{align*}
\|\nabla T_1f\|_{L^{\infty}(\rnp)}&=\|\nabla_x T_1f\|_{L^{\infty}(\rnp)}+\|\partial_{x_{n+1}}T_1f\|_{L^{\infty}(\rnp)}\\
&\leq c\big(\|I_1(  |\nabla f|)\|_{L^{\infty}(\rn)}+\|\mathcal{H}f\|_{L^{\infty}(\rnp)}\big)\\
&\leq c'\big(\|\nabla f\|_{L^M(\rn)}+\| f\|_{L^{\infty}(\rn)}\big)\leq c''\|f\|_{W^{1,M}(\rn)},
\end{align*}
for suitable constants $c, c', c''$, and 
for $f\in W^{1,M}(\rn)$.
 Here, we used the right-most side of the chain \eqref{2024-300}
in the first inequality, 
 and   Lemma~\ref{Sobolev-Riesz}, the second embedding in \eqref{2024-23}
and equation \eqref{2024-22} in the second inequality.
\\   Moreover, 
\begin{align*}
\|T_1f\|_{L^{\infty}(\rnp)}\leq \| T_1|f|\|_{L^{\infty}(\rn)}
=\|I_1 |f|\|_{L^{\infty}(\rn)}\
\leq c \| f\|_{L^M(\rn)} \leq c\|f\|_{W^{1,M}(\rn)},
\end{align*}
for a suitable constant $c$, and 
for $f\in W^{1,M}(\rn)$.
On the other hand, by  \eqref{2024-23},
\begin{align*}
\|T_1f\|_{W^{1,M_n}(\rn)}&=\|T_1f\|_{L^{M_n}(\rn)}+\|\nabla T_1f\|_{L^{M_n}(\rn)}\\
&\leq c\|I_1f\|_{L^{M_n}(\rn)}+c\|I_1(\nabla f)\|_{L^{M_n}(\rn)}\leq c'\|f\|_{W^{1,M}(\rn)},
\end{align*}
for some constants $c$ and $c'$, and for $f\in W^{1,M}(\rn)$.
\\ We have thus shown that 
$$\|  \nabla T_1f\|_{L^{\infty}(\rnp)} + \|T_1f\|_{W^{1,M_n}(\rn)} \leq   c\|f\|_{W^{1,M}(\rn)}$$
  for some constant $c$ and for $f\in W^{1,M}(\rn)$.
\\
 Finally, our assumptions on $M$ ensure, via the property \eqref{cont}, that $f\in C(\rn)$. Hence, a classical result from potential theory tells us that $T_1f \in C^\infty (\rnp) \cap C^1(\overline{\rnp})$.
\end{proof}


\begin{lem}\label{lem:L:B-norm}
Let $\alpha \in (0,n)$ and assume that $A$ is a  Young function satisfying~\eqref{conv-alpha}. Let 
$A_{\frac{n}{\alpha}}$ be its Sobolev conjugate of order $\alpha$ given by \eqref{A-alpha},
and let   $B_{A,\frac{n}{\alpha}}$ be the Young function defined by
 \eqref{E:B}. 
Then,
\begin{equation}\label{E:holder}
\|uv\|_{L^{A}(\rn)} \leq 8 \|u\|_{L^{A_{\frac{n}{\alpha}}}(\rn)} \|v\|_{L^{B_{A,\frac{n}{\alpha}}}(\rn)}
\end{equation}
for $u\in L^{A_{\frac{n}{\alpha}}}(\rn)$ and $v\in L^{B_{A,\frac{n}{\alpha}}}(\rn)$.
\end{lem}
\begin{proof}
Passing to inverse functions in equation \eqref{E:equivalence_B} yields
$$
\overline{H}_{A, \frac{n}{\alpha}}(A^{-1}(t)) \leq B_{A, \frac{n}{\alpha}}^{-1}(t) \leq 2\overline{H}_{A, \frac{n}{\alpha}}(A^{-1}(t)) \qquad \text{for $t\geq 0$.}
$$
Consequently, by the first inequality in \eqref{equivsob},  
$$
  \frac{1}{2} A_{\frac{n}{\alpha}}^{-1}(t) B_{A, \frac{n}{\alpha}}^{-1}(t)  \leq \widehat{A}_{\frac{n}{\alpha}}^{-1}(t) B_{A, \frac{n}{\alpha}}^{-1}(t)\leq 2 \widehat{A}_{\frac{n}{\alpha}}^{-1}(t)\overline{H}_{A, \frac{n}{\alpha}}(A^{-1}(t))= 2A^{-1}(t) \qquad \text{for $t\geq 0$.}
$$
Inequality \eqref{E:holder} hence follows as a special case of the inequality \eqref{holdergen}.
\end{proof}

\begin{lem}
    \label{BF}
    Under the same assumptions as in Lemma \ref{EF}, there exists a constant $c$ such that
\begin{equation}\label{2024-44}
\|uv\|_{L^{F}(\rn)} \leq c \|u\|_{L^{F_{\frac{n}{\alpha}}}(\rn)} \|v\|_{L^{B_{E,\frac{n}{\alpha}}}(\rn)}
\end{equation}
for $u\in L^{F_{\frac{n}{\alpha}}}(\rn)$ and $v\in L^{B_{E,\frac{n}{\alpha}}}(\rn)$.
\end{lem}
\begin{proof}
    Thanks to \eqref{2024-39}, the inequality \eqref{2024-44} is a special instance of   the inequality \eqref{holdergen}.
\end{proof}

\begin{lem}\label{L:lim0}
  Assume that $\alpha \in (0,n)$ and  $A$ is a finite-valued Young function.
  Let $D$ be the function defined as in \eqref{Dk}, with $\ell$ replaced with $\alpha$,  $E$ be given by \eqref{Enew}, and $E_{\frac{n}{\alpha}}$  be the Sobolev conjugate of $E$ defined as in \eqref{A-alpha}. Let $B_{E, \frac n\alpha}
$ be the function defined  by~\eqref{2024-31}. Assume that $\theta>0$. Then 
\begin{align}\label{2024-35}
   \lim_{\varepsilon \to 0^+} \,\sup_{\|u\|_{L^{E_{\frac{n}{\alpha}}}(\rn)}\leq \varepsilon} \left\| D(\theta|u|)\right\|_{L^{B_{E, \frac n\alpha}}(\rn)}=0. 
\end{align}
\end{lem}
\begin{proof}
Equation \eqref{2024-50} tells us that
$$
\overline{H}^{-1}_{E, \frac n\alpha}(D(t))=\widehat H_{E, \frac n\alpha}^{-1}(t) \qquad \text{for $t>0$} .
$$
Thus, thanks to Lemma~\ref{lem:L:B}, applied with $A$ replaced with $E$,
one has that
\begin{equation}\label{E:Bestimate}
B_{E, \frac n\alpha}(D(t)) \leq E(\overline{H}_{E, \frac n\alpha}^{-1}(D(t))=E(\widehat H_{E, \frac n\alpha}^{-1}(t))=\widehat E_{\frac n\alpha}(t) \qquad \text{for $t>0$.}
\end{equation}
Now, choose $\varepsilon >0$ such that $2\varepsilon\theta \leq 1$. 
By \eqref{Ak},~\eqref{E:Bestimate}, and equation ~\eqref{equivsob} with $A$ replaced with $E$,
\begin{align}\label{2024-34}
&\int_{\mathbb R^n} B_{E, \frac n\alpha}\bigg(\frac{D(\theta |u|)}{(2\varepsilon \theta)^{\frac{\alpha}{n-\alpha}}} \bigg)\,dx
=\int_{\rn\cap \{|u|>0\}} B_{E, \frac n\alpha}\left(\frac{A(\theta^{\frac{n}{n-\alpha}}|u|^{\frac{n}{n-\alpha}})}{\theta^{\frac{n}{n-\alpha}} (2\varepsilon)^{\frac{\alpha}{n-\alpha}}|u|}\right)\,dx\\ \nonumber
&\leq \int_{\mathbb R^n\cap \{|u|>0\}}B_{E, \frac n\alpha}\left(\frac{A\big(\frac{|u|^{\frac{n}{n-\alpha}}}{(2\varepsilon )^{\frac{n}{n-\alpha}}}\big)}{\frac{|u|}{2\varepsilon}}\right)\,dx
=\int_{\rn} B_{E, \frac n\alpha}\left(D\left(\frac{|u|}{2\varepsilon }\right)\right)\,dx\\ \nonumber
&\leq \int_{\mathbb R^n} \widehat{E}_{\frac n\alpha}\bigg(\frac{|u|}{2\varepsilon }\bigg)\,dx
\leq \int_{\mathbb R^n} E_{\frac{n}{\alpha}}\bigg(\frac{|u|}{\varepsilon}\bigg)\,dx
\leq 1.
\end{align}
This tells us that, 
 if  $u$ is such that $\|u\|_{L^{E_{\frac{n}{\alpha}}}(\rn)} \leq \varepsilon$, then
\begin{equation}\label{E:to_be_proved}
\left\| D(\theta|u|)\right\|_{L^{B_{E, \frac n\alpha}}(\rn)} \leq   \left(2\varepsilon \theta\right)^{\frac{\alpha}{n-\alpha}}.
\end{equation}
Hence, the inequality \eqref{E:to_be_proved} follows.
\end{proof}

\section{Proofs of Theorems  \ref{thm1}, \ref{thm2}, and \ref{thm:reg2}: harmonic functions }\label{second-order}

Our results about the second-order problem \eqref{eq:main-eq} are established in this section. As explained in Section \ref{intro}, although they are a special case of those concerning \eqref{eq:gen-eq}, a separate proof is offered for the readers' convenience.

{\color{black}

\begin{proof}[Proof of Theorem \ref{thm1}]
Set, for simplicity of notation,
 $$\mathbf{X}= V^{1,(E_\diamond,E_n)}(\rnp, \rn).$$ 
Given a function $f\in L^{E}(\rn)$, we define the 
 operator $\mathscr{L}$ by 
\begin{equation}\label{eq:fixed-pt-eq-proof}
\mathscr{L}u=T_1\big(\mathcal N(u)+f\big)
\end{equation}
for $u : \rnp \to \R$. We claim that, if $u  \in \mathbf{X}$ is a fixed point for $\mathscr{L}$, then it is 
 a weak solution to the problem \eqref{eq:main-eq}. Namely, 
 $u$ fulfills Definition \ref{def:weak-sol1}. To verify this claim, 
 we begin by observing that $\mathcal N(u)\in L^{E}(\rn)$. This inclusion follows from equation~\eqref{E:N} below. Also, recall that $E$ satisfies~\eqref{conv-Ealpha} with $\alpha=1$. Thus,
the argument from the proof of Lemma \ref{lem:reg} ensures that
one  can differentiate under the integral in the expression for $T_1(\mathcal N(u)+f)$.
Hence, given  $\varphi \in C^{\infty}_c(\overline{\rnp})$, owing to 
  Fubini's theorem and an integration by parts the following chain holds:
\begin{align*}
   \int_{\rnp}&\nabla u\cdot\nabla \varphi\, dxdx_{n+1}\\ \nonumber & =
   \int_{\rnp}\nabla T_1(\mathcal N(u)+f)\cdot\nabla \varphi \, dxdx_{n+1}\\
   &= \int_{\rnp} \bigg(\int_{\rn}\nabla K_1(x-y,x_{n+1})(\mathcal N(u)(y,0)+f(y))\,dy\bigg)\cdot\nabla \varphi (x,x_{n+1})\, dxdx_{n+1}\\
&=\int_{\rn}\bigg(\int_{\rnp}\nabla K_1(x-y,x_{n+1})\cdot\nabla \varphi (x,x_{n+1}) dxdx_{n+1}\bigg)(\mathcal N(u)(y,0)+f(y))\,dy\\
   &=\int_{\rn}(\mathcal N(u)(y,0)+f(y))\varphi(y,0)\,dy.
\end{align*}
Note that the last equality holds since, for each given $y\in \rn$, the kernel  $K_1$ is a distributional solution to the problem \begin{align*}
    \begin{cases}
\Delta K_1(x-y,x_{n+1})=0 \quad &\mbox{in}\,\,\rnp\\
\partial_{x_{n+1}}  K_1(x-y,0)=\delta_y(x)  \quad   &\mbox{on}\,\,\rn,
    \end{cases}
\end{align*}
where $\delta_y$ denotes the delta function centered at $y$.

The assertion about the existence of a weak solution to the problem \eqref{eq:main-eq}  will thus follow if we show that, 
 if $\|f\|_{L^E(\rn)}$ is sufficiently small, then the operator $\mathscr{L}$ admits a fixed point in $\mathbf{X}$.
\\
Let $\varepsilon>0$ to be chosen later and let $B^{\mathbf{X}}_{2\varepsilon}$ be the closed ball in $\mathbf{X}$, centered at the origin with radius $2\varepsilon$. First, we show that 
\begin{align}\label{2024-33}
\mathscr{L} : B^{\mathbf{X}}_{2\varepsilon} \to B^{\mathbf{X}}_{2\varepsilon},
\end{align}
provided that
 $\varepsilon$ is small enough. To verify this fact, fix \begin{align}\label{2024-36}
     u\in B^{\mathbf{X}}_{2\varepsilon}.
 \end{align} 
From Lemma \ref{lem:T-bounds} applied with $A$ replaced with $E$, we obtain
\begin{align}\label{eq:self-map}
\nonumber\|\mathscr{L}u\|_{\X}&=\|\nabla \,T_1(\mathcal N(u)+f)\|_{L^{E_\diamond}(\rnp)}+\|T_1(\mathcal N(u)+f)\|_{L^{E_n}(\rn)}\\
&\leq c(\|\mathcal N(u)\|_{L^{E}(\rn)}+\|f\|_{L^{E}(\rn)})
\end{align}
for some constant $c$.
Note that the assumptions of Lemma~\ref{lem:T-bounds} are satisfied, with $\ell =1$ and $A$ replaced by $E$, thanks to Lemmas~\ref{prop1} and~\ref{Enequiv}.
Let $B_{E, n}$ be the Young function introduced in~\eqref{E:B}, with $\alpha =1$. From assumption~\eqref{eq:N1} and the inequality~\eqref{E:holder},   applied with $A$ replaced with $E$ we obtain that
\begin{equation}\label{E:N}
\|\mathcal N(u)\|_{L^{E}(\rn)}
\leq c'\|u\|_{L^{E_n}(\rn)} \|D(|u|)\|_{L^{B_{E,n}}(\rn)}
\end{equation}
for some constant $c'$. Owing to the assumption \eqref{2024-36} and Lemma~\ref{L:lim0}, one can choose 
 $\varepsilon>0$  in such a way that 
$$
\left\| D(|u|)\right\|_{L^{B_{E,n}}(\rn)}\leq \frac{1}{2cc'}.
$$
 Therefore,  if \begin{align}\label{2024-47}
 \|f\|_{L^{E}(\rn)}\leq \varepsilon/c,
 \end{align}
 then equation \eqref{eq:self-map} implies that
 $\mathscr{L}(u) \in B^{\mathbf{X}}_{2\varepsilon}$, whence \eqref{2024-33} follows.
 \\ As a next step, we show that $\varepsilon$ can be chosen so small that the operator $\mathscr{L}$ is also a contraction on 
$B^{\mathbf{X}}_{2\varepsilon}$. An application 
of the assumption~\eqref{eq:N2} instead of~\eqref{eq:N1} and the same argument as above substantiate the following chain:
\begin{align}\label{eq:contraction}
\|\mathscr{L}u-\mathscr{L}v\|_{\X}&=\big\|\nabla\,T_1(\mathcal N(u)-\mathcal N(v))\big\|_{L^{E_\diamond}(\rnp)}+\|T_1(\mathcal N(u)-\mathcal N(v))\|_{L^{E_n}(\rn)}\\ \nonumber
&\leq c\|\mathcal N(u)-\mathcal N(v)\|_{L^{E}(\rn)}\\ \nonumber
&\leq \widehat c\|u-v\|_{L^{E_n}(\rn)}\left(\left\| D(\theta|u|)\right\|_{L^{B_{E,n}}(\rn)}+\left\|  D(\theta |v|)\right\|_{L^{B_{E,n}}(\rn)}\right)
\end{align}
for suitable constants $c$ and $\widehat c$, and 
for $u, v \in B^{\mathbf{X}}_{2\varepsilon}$.  Lemma~\ref{L:lim0} again ensures that
 $\varepsilon$ can be chosen so small that 
$$
\left\| D(\theta|u|)\right\|_{L^{B_{E,n}}(\rn)} \leq \frac{1}{4\widehat c}
$$
whenever $\|u\|_{L^{E_n}(\rn)}\leq 2\varepsilon$.  Hence, equation \eqref{eq:contraction} entails that
\begin{equation}
\|\mathscr{L}u-\mathscr{L}v\|_{\X}\leq \frac{1}{2}\|u-v\|_{L^{E_n}(\rn)}\leq \frac{1}{2}\|u-v\|_{\X}
\end{equation}
for $u,v\in B^{\mathbf{X}}_{2\varepsilon}(0)$. 
\\ Altogether, we have shown that if $\varepsilon$ is small enough, then the map \eqref{2024-33} is a contraction.
An application of the Banach fixed point theorem tells us that it admits a unique fixed point $u$. 
\\   The additional piece of information that $u \in C^\infty(\rnp)$ is a consequence of the fact that the kernel $K_1$ is smooth in $\rnp$ and, given any $L>0$ and $k,m \in \N\cup \{0\}$, satisfies the bound:
\begin{align}
    \label{2024-200}|\partial_{x_{n+1}}^k\nabla ^m_x K_1(x-y,x_{n+1})| \leq \frac {c} {(|x-y|^{2}+L^{2})^{\frac{n-1}{2}}} \quad \text{for $x, y \in \rn$ and $x_{n+1}\geq L$,}
\end{align}
for some constant $c=c(n,k,m,L)$. One has to use a standard argument outlined in the proof of Lemma \ref{lem:T-bounds}. 

 The inequality \eqref{nov4} is consequence of \eqref{2024-47} and of the fact that $u\in B^{\mathbf{X}}_{2\varepsilon}$.
\\
 It remains to prove the assertion about the sign of $u$.
Since the solution $u$ is a fixed point of the operator $\mathscr{L}$, it agrees with the limit 
 in $V^{1,(E_\diamond,E_n)}(\rnp,\rn)$  of the sequence  $\{u_j\}$ defined by
\begin{equation}\label{E:uj}
\begin{cases} & u_1=T_1f
\\
& u_{j+1}=T_1\big(\mathcal N(u_j)\big)+u_1 \quad \text{for} \quad j\in \mathbb   \mathbb N.
\end{cases}
\end{equation}
If   $f$ is positive a.e.\ and $\mathcal N(t)$ is positive for $t>0$, then obviously $u_1>0$.  By induction, we have that   $u_j\geq u_1>0$ for  $j\in \mathbb \N$. Hence, $u \geq u_1>0$, since the convergence in $V^{1,(E_\diamond,E_n)}(\rnp,\rn)$ implies  a.e.\ convergence, up to subsequences. 
\end{proof}

\begin{proof}[Proof of Theorem~\ref{thm2}]
Set, for brevity,
\begin{align}
    \label{Z}
    \Z=V^{1,(F_\diamond,F_n)}(\rnp,\rn).
\end{align}
In order to prove that $u\in \Z$, it suffices to show that the sequence $\{u_j\}$, defined by~\eqref{E:uj}, is a Cauchy sequence in $\Z$. Since $f\in L^{E\vee F}(\rn)$, then both $f \in L^{E}(\rn)$ and $f \in L^{F}(\rn)$.
  Let $\varepsilon$ and $c$
as in equation \eqref{2024-47}. 
An inspection of the proof of Theorem~\ref{thm1} shows that the smallness of $\|u_j\|_{L^{E_n}(\rn)}$ is guaranteed if
$\|f\|_{L^E(\rn)}< \sigma$ for a suitable  $\sigma\in (0,\tfrac \varepsilon c)$.
%
 In order to show that
$u_j \in \Z$ for each $j\geq 1$,  we argue by induction. 
 We have that $u_1 \in \Z$, 
by Lemma~\ref{lem:T-bounds}, applied with $A$ replaced with $F$.
Now assume that $j>1$.
Then,
\begin{align*}
\|u_j\|_{\Z}\leq \|u_{1}\|_{\Z} +\|T_1\big(\mathcal N(u_{j-1})\big)\|_{\Z}\leq \|u_{1}\|_{\Z} +c\|\mathcal N(u_{j-1})\|_{L^{F}(\rn)}.
\end{align*}
By \eqref{eq:N1}, the definition of $D$, the assumption \eqref{H} and the inequality \eqref{2024-44},  
\begin{align*}
\|\mathcal N(u_{j-1})\|_{L^{F}(\rn)}\leq
\|u_{j-1}\|_{L^{F_n}(\rn)} \left\| D(|u_{j-1}|)\right\|_{L^{B_{E,n}}(\rn)}
\leq \|u_{j-1}\|_{\Z} \left\| D(|u_{j-1}|)\right\|_{L^{B_{E,n}}(\rn)}.
\end{align*} 
Owing to Lemma~\ref{L:lim0},
$$
\left\| D(|u_{j-1}|)\right\|_{L^{B_{E,n}}(\rn)}<\infty,
$$
 thanks to the smallness of $\|u_{j-1}\|_{L^{E_n}(\rn)}$. 
 Altogether, we have that $u_j\in \Z$.
\\
Next, the following chain holds for  $j\geq 2$:
\begin{align*}
\|u_{j+1}-u_j\|_{\Z}
&=\big\|T_1[\mathcal N(u_j)-\mathcal N(u_{j-1})]\big\|_{\Z}
\leq c\|\mathcal N(u_j)-\mathcal N(u_{j-1})\|_{L^{F}(\rn)}\\
&\leq c' \|u_j-u_{j-1}\|_{L^{F_n}(\rn)} \left(\left\|D(\theta|u_j|)\right\|_{L^{B_{E,n}}(\rn)}+\left\| D(\theta|u_{j-1}|)\right\|_{L^{B_{E,n}}(\rn)}\right)
\end{align*}
for suitable constants $c$ and $c'$.
Here, we  used Lemma~\ref{lem:T-bounds}, with $A$ replaced with $F$, the assumption \eqref{eq:N2}, and Lemma \ref{BF}.
\\
We claim that 
$$
\left\| D(\theta|u_j|)\right\|_{L^{B_{E,n}}(\rn)} \leq \frac{1}{4c'} \quad \text{for } j\in \mathbb \N.
$$
By Lemma~\ref{L:lim0}, this claim  follows since, as observed above, $\|u_j\|_{L^{E_n}(\rn)}$ can be forced to be as small as we need for $j \in \mathbb \N$. 
%
We have thus proved that
$$
\|u_{j+1}-u_j\|_{\Z} \leq \frac{1}{2}\|u_j-u_{j-1}\|_{L^{F_n}(\rn)}
\leq \frac{1}{2}\|u_j-u_{j-1}\|_{\Z}  \quad \text{for } j\geq 2.
$$
Hence, $\{u_j\}$ is a Cauchy sequence in $ \Z$ converging to some function in $\Z$. This function must agree with $u$, whence \eqref{nov6} follows.
%
\end{proof}

}

\begin{proof}[Proof of Theorem~\ref{thm:reg2}]
The solution  under consideration to the problem \eqref{eq:main-eq}
is the limit in $V^{1,(E_\diamond,E_n)}(\rnp, \rn)$ of the sequence $\{u_j\}$ introduced in~\eqref{E:uj}. Define the space
$$\Y=\big\{u\in C^1(\overline{\rnp}): \, \tr \in W^{1,M_n}(\rn)\}.$$
Then $\Y$ is a Banach space endowed with the norm
$$\|u\|_{\Y}=\|u\|_{C^1(\overline{\rnp})}+  \|\tr \|_{W^{1,M_{n}}(\rn)}.$$
 Our goal is to show that  $\{u_j\}$ is Cauchy sequence in  $\Y$, provided that
$\|f\|_{W^{1,M}(\rn)}$ is sufficiently small. It will then follow that $u\in \Y$ and that  \eqref{nov10} holds.
\\  
Let $c_0$ and $t_0$ be as in the inequality \eqref{E:convexity}. Owing to our assumptions on $M$, there exists a constant $c_1$ such that
\begin{align}\label{2024-60inf}
\|v\|_{L^\infty (\rn)}+ \|\nabla v\|_{L^\infty (\rn)}\leq c_1\|v\|_{\Y}
\end{align}
for  $v\in \Y$.
Fix $\sigma >0$. Let $\varepsilon>0$ be such that 
\begin{align}\label{2024-59}
   \varepsilon< \frac 1\sigma.
\end{align} 
Assume that $v$ is any function in $\Y$ such that
\begin{align}\label{2024-60'}
\|v\|_{\Y}\leq \varepsilon.
\end{align}
Let $Q_1, Q_2 : [0, \infty) \to [0, \infty)$ be the functions  defined as 
$$Q_1(t)=M_{n} (2 \phi^{-1}(t)) \quad \text{and} \quad Q_2=M_{n} (2 \psi^{-1}(t)) \quad \text{for $t\geq 0$.}$$
The function $M_n(t)$ is finite-valued for small $t$ and admits a classical inverse. Moreover
\begin{equation*}
Q_1^{-1}(t)
Q_2^{-1}(t)=  \phi\big(\tfrac 1 2 M_n^{-1}(t)\big) \psi\big(\tfrac 1 2 M_n^{-1}(t)\big) \leq
D(\tfrac 12 M_n^{-1}(t)) \qquad \text{for $t \in (0, t_1)$,}
\end{equation*}
for a sufficiently small $t_1>0$.
{\color{black}
By Lemma \ref{lem:MnD} and Lemma \ref{lem:L:B} applied with $A$ replaced with $M$, the number $t_1$ can be chosen so small that
\begin{equation}\label{E:BD}
D(\tfrac 12 M_n^{-1}(t))\leq \overline H_{M,n}(M^{-1}(t)) \leq B_{M,n}^{-1}(t)\qquad
  \text{for $t\in (0,t_1)$, }
\end{equation}
where $\overline H_{M,n}$  is defined as in \eqref{2024-31} and $B_{M,n}$ is defined as in \eqref{2024-32} with $E$ replaced with $M$. Hence,
\begin{align}
    \label{2024-61}
    Q_1^{-1}(t)
Q_2^{-1}(t) \leq  B_{M,n}^{-1} (t)\quad \text{for $t\in (0,t_1)$.}
\end{align}
%
 Let $$k>\max\Big\{2,\frac {c_1}{t_0}\Big\},$$ where $t_0$ is the constant from \eqref{E:convexity}.
By \eqref{E:convexity}  and \eqref{2024-60'},
\begin{align}\label{Y1}
\int_{\rn}Q_1\bigg(\frac{\phi(\sigma|v|/k)}{c_0(\varepsilon\sigma)^{\gamma}}\bigg)dx
&\leq   \int_{\rn}Q_1\bigg(\phi\bigg(\frac{|v|}{k\varepsilon}\bigg)\bigg)dx =     \int_{\rn} M_n\bigg(2\phi^{-1}\bigg(\phi \bigg(\frac{|v|}{k\varepsilon}\bigg)\bigg)\bigg)dx
\\ \nonumber
&\leq \int_{\rn}M_n\bigg(\frac{|v|}{\varepsilon}\bigg)dx \leq 1.
\end{align}
On the other hand, 
\begin{align}\label{Y2}
\int_{\rn}Q_2\bigg( \psi\bigg(\frac{\sigma|v|}{k}\bigg)\bigg)\,dx
\leq  \int_{\rn}Q_2\bigg( \psi\bigg(\frac{|v|}{k\varepsilon}\bigg)\bigg)dx \leq \int_{\rn}M_n\bigg(\frac{|v|}{\varepsilon}\bigg)dx \leq 1.
\end{align}
Hence,
\begin{equation}\label{control-nonl}\|\phi(\sigma |v|/k)\|_{L^{Q_1}(\rn)}\leq  c_0 (\varepsilon\sigma)^{\gamma}\quad\mbox{and}\quad \|\psi(\sigma |v|/k)\|_{L^{Q_2}(\rn)}\leq 1.\end{equation} }
Notice that $u_1 \in \Y$ since 
$$
\|u_1\|_{\Y}=\|T_1 f\|_{\Y} \leq c\|f\|_{W^{1,M}(\rn)}
$$
for some constant $c$,
where the inequality is a consequence of Lemma~\ref{lem:reg}. This entails the smallness of  $\|u_1\|_{\Y}$  when $\|f\|_{W^{1,M}(\rn)}$ is small enough.
We next prove by induction that the same smallness property is enjoyed by $\|u_j\|_{\Y}$, uniformly in $j\in \N$.
Assume that \begin{align}\label{2024-60}
\|u_j\|_{\Y}\leq \varepsilon
\end{align}
for some $j\in \mathbb \N$.
 \\
By Lemma~\ref{lem:reg} again, 
\begin{align}\label{2024-72}
\|u_{j+1}\|_{\Y} 
=\|T_1(\mathcal N(u_j))+T_1 f\|_{\Y}
\leq c\|\mathcal N(u_j)\|_{W^{1,M}(\rn)}+c\|f\|_{W^{1,M}(\rn)}
\end{align}
for some constant $c$. Moreover,
\begin{equation}\label{eq:Nuj}
\|\mathcal N(u_j)\|_{W^{1,M}(\rn)}
=\|\mathcal N(u_j)\|_{L^M(\rn)} +\|\mathcal N'(u_j)\nabla u_j\|_{L^M(\rn)}.
\end{equation}
Assume, in addition, that $\varepsilon$ is such that
\begin{align}
    \label{2024-65}
    \varepsilon <   \frac {M_n^{-1}(t_1)}{{2c_1}}.
\end{align}
Therefore,
\begin{align}\label{2024-70}
\|\mathcal N(u_j)\|_{L^{M}(\rn)}
\leq c\|u_j\|_{L^{M_n}(\rn)} \|D(|u_j|)\|_{L^{B_{M,n}}(\rn)} 
\leq c'\varepsilon^{n'-1}\|u_j\|_{\Y},
\end{align}
for some constants $c, c'$ independent of $\varepsilon$. Here, the first inequality follows from~\eqref{eq:N1} and \eqref{E:holder} with $A$ replaced with $M$, whereas the second inequality is a consequence of the estimate~\eqref{E:BD} and of the first inequality in~\eqref{control-nonl} applied with $\sigma=k$ and with $\phi$ replaced with $D$. Note that \eqref{control-nonl} holds with $\gamma=n'-1$, $c_0=1$ and $t_0=\infty$ in the present situation, since, thanks to~\eqref{E:D-epsilon}, equation
\eqref{E:convexity} holds with $\phi = D$.
\\
Next, notice that  equation ~\eqref{eq:N2} implies that
\begin{align}\label{2024-81}
    |\mathcal N'(t)| \leq c \theta D(\theta t) \quad \text{for a.e. $t>0$,}
\end{align}
 and for some constant $c$.
Hence, an analogous chain as in \eqref{2024-70}, with \eqref{control-nonl} now applied with $\sigma =  k\theta$, yields
\begin{align}\label{2024-71}
   \| \mathcal N'(u_j)\nabla u_j\|_{L^M(\rn)}
\leq c\|\nabla u_j\|_{L^{M_n}(\rn)} \|D(\theta|u_j|)\|_{L^{B_{M,n}}(\rn)}
\leq c'\varepsilon^{n'-1} \|u_j\|_{\Y},
\end{align}
for some constants $c, c'$ independent of $\varepsilon$.
From \eqref{2024-72}, \eqref{2024-70}, and \eqref{2024-71} we deduce that 
\begin{align}\label{2024-73}
\|u_{j+1}\|_{\Y}\leq \varepsilon,
\end{align}
provided that $\varepsilon$ is sufficiently small, and 
$\|f\|_{W^{1,M}(\rn)}\leq  \frac{\varepsilon}{2c}$, where  $c$ is the constant appearing in \eqref{2024-72}.   As a consequence, since the inequalities in \eqref{control-nonl}
follow from \eqref{2024-60},  they hold for every $j \in \N$.
\\
 The next step consists of a bound for $\|u_{j+1}-u_j\|_{\Y}$ for $j \in \N$.
 We begin
with an estimate for the norm $\|u_{j+1}-u_j\|_{L^{M_n}(\rn)}$. An application of Lemma \ref{Sobolev-Riesz}, with $A$ replaced with $M$, yields, via the same steps as in \eqref{2024-70},
\begin{align}\label{E:mn}
\|u_{j+1}-u_j\|_{L^{M_n}(\rn)}
 &=\|T_1(\mathcal N(u_j)-\mathcal N(u_{j-1}))\|_{L^{M_n}(\rn)}\\
\nonumber &= c\|I_1(\mathcal N(u_j)-\mathcal N(u_{j-1}))\|_{L^{M_n}(\rn)}\\
\nonumber
&\leq c'\|\mathcal N(u_j)-\mathcal N(u_{j-1})\|_{L^M(\rn)}\\
\nonumber
&\leq c''\|u_j-u_{j-1}\|_{L^{M_n}(\rn)} \big(\|D(\theta|u_j|)\|_{L^{B_{M,n}}(\rn)} +\|D(\theta|u_{j-1}|)\|_{L^{B_{M,n}}(\rn)}\big)\\
& \nonumber \leq c'''  \varepsilon^{n'-1}\|u_j-u_{j-1}\|_{L^{M_n}(\rn)} 
\end{align}
for sufficiently small $\varepsilon$ and for suitable constants $c, c', c{''}, c{'''}$. The smallness of 
 $\varepsilon$ hence    ensures that
$$
\|u_{j+1}-u_j\|_{L^{M_n}(\rn)}\leq \frac{1}{2}\|u_{j}-u_{j-1}\|_{L^{M_n}(\rn)},
$$
for $j \in \N$,
whence
\begin{equation}\label{E:decay-mn}
\|u_{j+1}-u_j\|_{L^{M_n}(\rn)}\leq 2^{-j}.
\end{equation}
An application of Lemma \ref{lem:reg} tells us that  
\begin{align}\label{E:c2}
\|u_{j+1}-u_j\|_{\Y}
&\leq \overline c\|\mathcal N(u_j)-\mathcal N(u_{j-1})\|_{W^{1,M}(\rn)}\\
\nonumber
&= \overline c\|\mathcal N(u_j)-\mathcal N(u_{j-1})\|_{L^{M}(\rn)}+ \overline c\|\nabla(\mathcal N(u_j)-\mathcal N(u_{j-1}))\|_{L^{M}(\rn)},
\end{align}
  for some constant $\overline c$.
From the estimates ~\eqref{E:mn} and ~\eqref{E:decay-mn} one infers that
\begin{align}\label{2024-80}
    \|\mathcal N(u_j)-\mathcal N(u_{j-1})\|_{L^{M}(\rn)} \leq 2^{-j}
\end{align}
for $j \in \N$,
if $\varepsilon$ is small enough.
Furthermore,
\begin{align}\label{E:lm}
\|\nabla(\mathcal N(u_j)&-\mathcal N(u_{j-1}))\|_{L^{M}(\rn)}\\
&\leq \|\nabla(u_j-u_{j-1}) \mathcal N'(u_j)\|_{L^M(\rn)} +\|\nabla u_{j-1} (\mathcal N'(u_j)-\mathcal N'(u_{j-1}))\|_{L^M(\rn)}. \nonumber
\end{align}
Thanks to \eqref{2024-81} and to analogous steps as above,
\begin{align}\label{2024-102}
\|\nabla(u_j-u_{j-1}) \mathcal N'(u_j)\|_{L^M(\rn)}
&\leq c\|\nabla(u_j-u_{j-1})D(\theta|u_j|)\|_{L^{M}(\rn)}\\
&\leq \nonumber c'\|\nabla(u_j-u_{j-1})\|_{L^{M_n}(\rn)} \|D(\theta |u_j|)\|_{L^{B_{M,n}}(\rn)}\\
&\leq c'' \nonumber   \varepsilon^{n'-1}\|u_j-u_{j-1}\|_{\Y} 
\\ \nonumber
&\leq \frac{1}{2 \overline c}\|u_j-u_{j-1}\|_{\Y},
\end{align}
 for some constants $c, c', c{''}$,
provided that $\varepsilon$ is sufficiently small.  Here, $\overline c$ denotes the constant from~\eqref{E:c2}.
\\ Finally, the same arguments that 
yield \eqref{control-nonl}, with $\varepsilon=2^{-j}$ and a suitable choice of $\sigma$, and \eqref{E:decay-mn} imply that
\begin{equation}\label{Yi}
\|\phi(\theta|u_j-u_{j-1}|)\|_{ L^{Q_1}(\rn)}\leq  c 2^{-j\gamma},
\end{equation}
  for some constant $c$.
By making use of the assumption \eqref{eq:N3}, the inequality~\eqref{2024-61},  the inequality~\eqref{holdergen} in the version under $L^\infty$-bounds for trial functions, as well as \eqref{Yi} and \eqref{control-nonl}, we deduce that
\begin{align}\label{2024-85}
\|\nabla u_{j-1}& (\mathcal N'(u_j)-\mathcal N'(u_{j-1}))\|_{L^M(\rn)}\\ \nonumber
&\,\,\leq c\|\nabla u_{j-1}\|_{L^{M_n}(\rn)} \|\mathcal N'(u_j)-\mathcal N'(u_{j-1})\|_{L^{B_{M,n}}(\rn)}\\ \nonumber
&\,\,\leq c'\|u_{j-1}\|_{\Y} (\|\phi(\theta|u_j-u_{j-1}|)\psi(\theta|u_j|)\|_{L^{B_{M,n}}(\rn)}+\|\phi(\theta|u_j-u_{j-1}|)\psi(\theta|u_{j-1}|)]\|_{L^{B_{M,n}}(\rn)}\\ \nonumber
&\,\,\leq c''\varepsilon
\|\phi(\theta|u_j-u_{j-1}|)\|_{ L^{Q_1}(\rn)}
\left( \| \psi(\theta|u_j|)\|_{ L^{Q_2}(\rn)} + \| \psi(\theta|u_{j-1}|)\|_{ L^{Q_2}(\rn)} \right) \\ \nonumber
& \,\, \leq c'''\varepsilon
2^{-j\gamma},
\end{align}
  for suitable constants $c, c', c{''}, c{'''}$.
Altogether, we conclude that
$$
\|u_{j+1}-u_j\|_{\Y}
\leq \frac{1}{2}\|u_{j}-u_{j-1}\|_{\Y}+c2^{-j\gamma}
$$
for   some constant $c$ and for $j\in \N$.
Iterating this inequality yields
\begin{align}
    \label{2024-107}
    \|u_{j+1}-u_j\|_{\Y}
\leq \frac{1}{2^{j-1}} \|u_2-u_1\|_{\Y} +c2^{-j
} \sum_{i=0}^{j-2} 2^{ (1-\gamma)i}
\leq  c'(2^{-j}+ 2^{-j\gamma}),
\end{align}
  for some constants $c, c'$.
Since $\gamma>0$, it follows that $\{u_j\}$ is a Cauchy sequence in $\Y$. The proof is complete.
\end{proof}

\section{Proofs of Theorems  \ref{thm:g}--\ref{thm:gen-eq-reg}: polyharmonic functions }
\label{higher-order}
Our approach to the higher-order problem \eqref{eq:gen-eq} is similar to that exposed 
 in the previous section. 
In the proofs offered below, we shall thus mainly focus on the steps requiring major variants. 

\begin{proof}[Proof of Theorem \ref{thm:g}] 
Consider the operator $\mathscr{L}$ 
defined as
\begin{equation}\label{eq:FP_poly-proof}
\mathscr{L}u=T_\ell\big(\mathcal N(u)+f\big)
\end{equation}
for $u : \rnp \to \R$. 
 A fixed point   $u\in V^{\ell,(E_\diamond,E_{\frac{n}{\ell}})}(\rnp, \rn)$ for this map is a
    solution to the problem~\eqref{eq:gen-eq}. To verify this assertion, we need to show that 
\begin{equation}\label{weak_sol_for}\int_{\rnp}\nabla^mu\cdot\nabla^m\varphi \, dx dx_{n+1} =\int_{\rn}(\mathcal N(u)(x,0)+f(x))\varphi (x,0)\,dx\end{equation}
for every      $\varphi \in C^{\infty}_c(\overline{\rnp})$ such that $\partial_{x_{n+1}}\Delta^k\varphi(\cdot,0)=0$  for  $k=0,\dots m-2$.
Fixing  $y \in \rn$, the kernel $K_{\ell}$ of the operator $T_{\ell}$ is a distributional solution to the problem  
\begin{align*}
    \begin{cases}
\Delta^m K_{\ell}(x-y,x_{n+1})=0 \quad & \text{in $\rnp$} \\
\partial_{x_{n+1}}\Delta^{k}K_{\ell}
(x-y,0)=0
\quad 
& \text{on $\rn$, for $k=0,1,...,m-2$}\\
(-1)^m
\partial_{x_{n+1}}\Delta^{m-1}K_{\ell}
(x-y,0)=\delta_y(x)    \quad & \text{on $\rn$.}
    \end{cases}
\end{align*}
Hence, if $\varphi$ is as above, then
\begin{align}\label{IBPpoly}
\int_{\rn}K_{\ell}(x-y,0)\partial_{x_{n+1}} &\Delta^{m-1}\varphi(x,0)-(-1)^{m} \varphi(x,0)\delta_y(x)
\,dx\\ 
\nonumber
&+\int_{\rnp}K_{\ell}(x-y, x_{n+1})\Delta^m\varphi(x,  x_{n+1}) \, dx  dx_{n+1}=0.
\end{align}
An iterated integration by parts and the use of the conditions
$\partial_{x_{n+1}}\Delta^k\varphi(\cdot,0)=0$ for $k=0, \dots ,m-2$ yield
\begin{align}
\label{claim}\int_{\rnp} K_{\ell}(x-y, x_{n+1})&\Delta^m\varphi (x, x_{n+1})\, dxdx_{n+1} \\ \nonumber & =  (-1)^m\int_{\rnp}\nabla^m K_{\ell}(x-y, x_{n+1})\cdot\nabla^m \varphi (x, x_{n+1})\, dxdx_{n+1} 
 \\ \nonumber & \quad -\int_{\rn}K_{\ell}(x-y,0)\partial_{x_{n+1}}(\Delta^{m-1}\varphi)(x,0) \,dx.
\end{align}
Coupling \eqref{IBPpoly} with \eqref{claim}  produces
$$\int_{\rnp}\nabla^m K_{\ell}(x-y,x_{n+1})\cdot\nabla^m  \varphi (x,x_{n+1})\, dxdx _{n+1}=\varphi(y,0).$$
As a consequence, thanks to equation \eqref{2024-135}, the following chain holds:
\begin{align*}
   \int_{\rnp}\nabla^m u & \cdot\nabla^m \varphi \, dxdx_{n+1} \\ \nonumber &=
   \int_{\rnp}\nabla^m T_{\ell}(\mathcal N(u)+f)\cdot\nabla^m \varphi \, dxdx_{n+1} \\
   &= \int_{\rnp} \bigg(\int_{\rn}\nabla^m K_\ell(x-y,x_{n+1})(\mathcal N(u)(y,0)+f(y))dy\bigg)\cdot\nabla^m \varphi (x,x_{n+1})\, dxdx_{n+1} \\
&=\int_{\rn}\bigg(\int_{\rnp}\nabla^m K_\ell(x-y,x_{n+1})\cdot\nabla^m \varphi (x,x_{n+1}) dxdx_{n+1}\bigg)(\mathcal N(u)(y,0)+f(y))dy\\
   &=\int_{\rn}(\mathcal N(u)(y,0)+f(y))\varphi(y,0)dy,
\end{align*}
namely \eqref{weak_sol_for}.
\\
Our aim is now to show that
  the integral equation \eqref{eq:FP_poly-proof} admits a unique   fixed point $u$, provided that the norm $\|f\|_{L^E(\rn)}$ is sufficiently small. 
Let us set, for simplicity of notation, $$\mathbf{X}= V^{\ell,(E_\diamond,E_{\frac{n}{\ell}})}(\rnp, \rn).$$ 
   To prove that
\begin{align}\label{2024-86}
\mathscr{L} : B^{\mathbf{X}}_{2\varepsilon} \to B^{\mathbf{X}}_{2\varepsilon},
\end{align}
for sufficiently small
 $\varepsilon$, one can argue as in the proof of Theorem~\ref{thm1}, and  obtain that
\begin{align}\label{2024-87}
\|\mathscr{L}u\|_{\X}
\leq c\big(\|u\|_{L^{E_{\frac{n}{\ell}}}(\rn)} \|D(|u|)\|_{L^{B_{E,\frac n\ell}}(\rn)}+\|f\|_{L^{E}(\rn)}\big),
\end{align}
  for some constant $c$.
Thanks to Lemma~\ref{L:lim0}, there exists $\varepsilon>0$ such that 
$$
 \|D(|u|)\|_{L^{B_{E,\frac n\ell}}(\rn)}\leq \frac{1}{2c}
$$
whenever $\|u\|_{\X}\leq 2\varepsilon$. Thus, if $\|f\|_{L^{E}(\rn)}\leq \varepsilon/c$, then   equation \eqref{2024-87} ensures that $\|\mathscr{L}(u)\|_{\X}\leq 2\varepsilon$. Hence, \eqref{2024-86} follows.
\\ As for the contraction property of the operator $\mathscr{L}$, one similarly finds that
\begin{align*}
\|\mathscr{L}u-\mathscr{L}v\|_{\X}
\leq c\|u-v\|_{L^{E_{\frac{n}{\ell}}}(\rn)}\big(\left\| D(\theta|u|)\right\|_{L^{B_{E, \frac n\ell}}(\rn)}+\left\|  D(\theta |v|)\right\|_{L^{B_{E, \frac n\ell}}(\rn)}\big),
\end{align*}
  for some constant $c$.
By Lemma~\ref{L:lim0} again, there exists $\varepsilon>0$ such that
$$
\left\| D(\theta|u|)\right\|_{L^{B_{E, \frac n\ell}}(\rn)}+\left\|  D(\theta |v|)\right\|_{L^{B_{E, \frac n\ell}}(\rn)} \leq \frac 1{2c}
$$
provided that  $\|u\|_{L^{E_{\frac{n}{\ell}}}(\rn)}\leq 2\varepsilon$ and $\|v\|_{L^{E_{\frac{n}{\ell}}}(\rn)}\leq 2\varepsilon$.  Hence,
\begin{equation*}
\|\mathscr{L}u-\mathscr{L}v\|_{\X}\leq \frac{1}{2}\|u-v\|_{L^{E_{\frac{n}{\ell}}}(\rn)}\leq \frac{1}{2}\|u-v\|_{\X}
\end{equation*}
for $u,v\in B^{\mathbf{X}}_{2\varepsilon}$. An application of the Banach fixed point theorem yields the existence of a unique fixed point for the map \eqref{2024-86}. 
\\  Since the kernel $K_\ell$ is smooth in $\rnp$ and, given any $L>0$ and $k,m \in \N\cup \{0\}$, 
\begin{align}
    \label{2024-201}|\partial_{x_{n+1}}^k\nabla ^m_x K_\ell(x-y,x_{n+1})| \leq \frac {c} {(|x-y|^{2}+L^{2})^{\frac{n-\ell}{2}}} \quad \text{for $x, y \in \rn$ and $x_{n+1}\geq L$,}
\end{align}
for some constant $c=c(n,\ell,k,m,L)$,   the argument sketched in the proof of Lemma \ref{lem:T-bounds} ensures that  $u \in C^\infty(\rnp)$.

Being the fixed point of the map \eqref{2024-86},
 the function $u$ is the limit in $V^{\ell,(E_\diamond,E_{\frac{n}{\ell}})}(\rnp,\rn)$  of the sequence $\{u_j\}$ defined as
\begin{equation}\label{Picard}
\begin{cases} & u_1=T_\ell f
\\
& u_{j+1}=T_\ell\big(\mathcal N(u_j)\big)+u_1 \quad \text{for} \quad j\in \mathbb \N.
\end{cases}
\end{equation}
If   $f>0$ is positive a.e.\ and $\mathcal N(t)>0$ is positive for $t>0$, then  $u_1>0$ and,  by induction,   $u_j\geq u_1>0$ for  $j\in \mathbb \N$. Thus, $u \geq u_1>0$, since the convergence in $V^{\ell,(E_\diamond,E_{\frac n\ell})}(\rnp,\rn)$ implies the   convergence a.e. \ of a subsequence of $\{u_j\}$.
\end{proof}

\begin{proof}[Proof of Theorem~\ref{thm4}]
The proof is analogous to that of Theorem \ref{thm2}. One has just to define the space $\Z$ as
$$\Z=V^{\ell,(F_\diamond,F_{\frac{n}{\ell}})}(\rnp,\rn),$$ 
to replace the first-order gradient $\nabla$ with the operator $\nabla^\ell$, and the functions $F_n$ and $B_{E,n}$ with $F_{\frac n\ell}$ and $B_{E,\frac n\ell}$. Of course, the sequence $\{u_j\}$ is now defined as in \eqref{Picard}. We skip the details for brevity.
\end{proof}

\begin{proof}[Proof of Theorem~\ref{thm:gen-eq-reg}.]   
Set $w=\Delta^{m-1}u$. Then, $w$ solves the second-order problem 
\begin{align}
\begin{cases}
\Delta w=0\,\,\mbox{ in }\,\,\mathbb{R}^{n+1}_{+}\\
(-1)^{m}\partial_{x_{n+1}}w=\mathcal N(u)+f\,\,\mbox{ on }\,\,\partial\mathbb{R}^{n+1}_{+}.
\end{cases}
\end{align}
If $\mathcal N(u)$ has the same regularity as $f$, namely, if \begin{align}
    \label{2024-270}
    \mathcal N(u)\in W^{1,M}(\rn),
\end{align} then, by Lemma~\ref{lem:reg}, $w\in C^1(\overline{\rnp})$.  Consequently, equation \eqref{nov10k}
holds by the classical elliptic regularity theory \cite{ADN}. It thus suffices to establish equation \eqref{2024-270}.
\\ Consider the sequence $\{u_j\}$ of functions $u_j : \rn \to \R$ defined as 
\begin{equation}\label{E:ujell}
\begin{cases}  \,u_1=I_{\ell}f
\\
 u_{j+1}=I_{\ell}(\mathcal N(u_{j}))+u_1 \quad \text{for} \quad j\in \mathbb N.
\end{cases}
\end{equation}
We shall prove that $\{u_j\}$ is a Cauchy sequence in the Orlicz-Sobolev space $W^{1,M_{\frac{n}{\ell}}}(\rn)$, provided that the norm $\|f\|_{W^{1,M}(\rn)}$ is small enough. 
\\  Thanks to   \eqref{Mk}, there exists a constant $c_2$ such that
\begin{align}
    \label{2024-90}
    \|v\|_{L^\infty(\rn)}\leq c_2  \|v\|_{L^{M_{\frac{n}{\ell}}}(\rn)}
\end{align}
for $v\in L^{M_{\frac{n}{\ell}}}(\rn)$.
Fix $\sigma >0$. Let  
 $\varepsilon>0$ be such that
 $\varepsilon < \frac 1\theta$. 
 Let $v \in L^{M_{\frac{n}{\ell}}}(\rn)$ be any function such that $$\|v\|_{L^{M_{\frac{n}{\ell}}}(\rn)}\leq \varepsilon.$$
\\
Define the functions $Q_1, Q_2 : [0, \infty) \to [0, \infty)$ as
$Q_1(t)=M_{n/\ell} (2\phi^{-1}(t))$ and $Q_2(t) =M_{n/\ell}( 2 \psi^{-1}(t))$ for $t \geq 0$.
Then,
 $$Q_1^{-1}(t)
Q_2^{-1}(t)=D(\tfrac 12 M_{\frac n\ell}^{-1}(t)) \qquad \text{for $t\geq 0$.}$$
Moreover, 
by Lemma \ref{lem:MnD} and Lemma \ref{lem:L:B} applied with $A$ replaced with $M$, there exists $t_1$ such that
\begin{equation}\label{E:db2}
D(\tfrac 12 M_{\frac n\ell}^{-1}(t))\leq \overline H_{M,\frac n\ell}(M^{-1}(t)) \leq B_{M,\frac n\ell}^{-1}(t)\qquad
  \text{for $t\in (0,t_1)$, }
\end{equation}
where $\overline H_{M,\frac n\ell}$  is defined as in \eqref{2024-31} and $B_{M,\frac n\ell}$ is defined as in \eqref{2024-32} with $E$ replaced with $M$. Therefore,  
\begin{align}
    \label{2024-61-b}
    Q_1^{-1}(t)
Q_2^{-1}(t) \leq  B_{M,\frac n\ell}^{-1} (t)\quad \text{for $t\in (0,t_1)$.}
\end{align}
%
 Let $$k>\max\Big\{2,\frac {c_2}{t_0}\Big\},$$ where $t_0$ is the constant appearing in \eqref{E:convexity}.  Analogous chains as in \eqref{Y1} and \eqref{Y2} tell us that 
\begin{align}\label{2024-92}
   \|\phi(\sigma |v|/k)\|_{L^{Q_1}(\rn)}\leq  c_0 (\varepsilon\sigma)^{\gamma} \quad\mbox{and}\quad \|\psi(\sigma |v|/k)\|_{L^{Q_2}(\rn)}\leq 1,
\end{align}
where $c_0$ is the constant from \eqref{E:convexity}.
\\  In particular, if $\phi=D$ and $\psi =1$, then the inequality \eqref{E:convexity} holds with $\gamma=\frac{\ell}{n-\ell}$. Hence, 
  if 
\begin{align*}
    \varepsilon <  \frac {M_{\frac{n}{\ell}}^{-1}(t_1)}{2c_2}
\end{align*}
then equation \eqref{2024-92}  combined with~\eqref{E:db2} implies that
\begin{align}\label{2024-93}
   \|D(\sigma|v|/k)\|_{L^{B_{M,
   \frac{n}{\ell}}}(\rn)} \leq c \sigma^{\frac{\ell}{n-\ell}} \varepsilon^{\frac{\ell}{n-\ell}}. 
\end{align}
From Lemma \ref{Sobolev-Riesz}, applied with $A$ replaced with $M$, we deduce that $u_1\in W^{1,M_{\frac{n}{\ell}}}(\rn)$. Moreover, 
$\|u_1\|_{W^{1,M_{\frac{n}{\ell}}}(\rn)}\leq \varepsilon$ provided that  
$\|f\|_{  W^{1,M}(\rn)}$ is small enough.
\\ Next, assume that 
\begin{align}
    \label{2024-90bis}
    \|u_j\|_{W^{1,M_{\frac{n}{\ell}}}(\rn)}\leq \varepsilon 
\end{align}
for some $j \in \N$.
We have that
\begin{align}\label{2024-95}
 \|u_{j+1}\|_{W^{1,M_{\frac{n}{\ell}}}(\rn)} 
& \leq c\|\mathcal N(u_j)\|_{W^{1,M}(\rn)}+c\|f\|_{W^{1,M}(\rn)}
\\ \nonumber & = c\|\mathcal N(u_j)\|_{L^M(\rn)} +c\|\mathcal N'(u_j)\nabla u_j\|_{L^M(\rn)}+c\|f\|_{W^{1,M}(\rn)},
\end{align}
  for some constant $c$.
Through an appropriate choice of $\sigma$ in   \eqref{2024-93},
an analogous chain as in \eqref{2024-70} tells us  that there exists a constant $c$ such that
\begin{align}\label{2024-97}\|\mathcal N(u_j)\|_{L^M(\rn)}&
\leq 
 c \varepsilon^{  \frac{\ell}{n-\ell}} \|u_j\|_{W^{1,M_{\frac{n}{\ell}}}(\rn)} .
\end{align}
Also, similarly to \eqref{2024-71},
\begin{align}\label{2024-98}
\|\mathcal N'(u_j)\nabla u_j\|_{L^M(\rn)}
\leq  c \varepsilon^{ \frac{\ell}{n-\ell}} \|u_j\|_{W^{1,M_{\frac{n}{\ell}}}(\rn)},
\end{align}
 for some constant $c$.
Therefore, an  induction argument ensures that 
$\|f\|_{W^{1,M}(\rn)}$ can be chosen small enough for equation \eqref{2024-90bis}
to hold for every $j \in \N$.
Via a suitable choice of $\sigma$ in \eqref{2024-93}, the same steps as in equation \eqref{E:mn} yield:
\begin{align}\label{E:mn2}
\|u_{j+1}-u_j\|_{L^{M_{\frac{n}{\ell}}}(\rn)}
\leq c 
\|\mathcal N(u_j)-\mathcal N(u_{j-1})\|_{L^M(\rn)}
\leq c'
\varepsilon^{ \frac{\ell}{n-\ell}}\|u_j-u_{j-1}\|_{L^{M_{\frac{n}{\ell}}}(\rn)}
\end{align}
 for suitable constants $c, c'$,
provided that $\varepsilon$ is sufficiently small. Consequently, if $\varepsilon$ is properly chosen, one deduces that
\begin{align} \label{2024-99}
\|u_{j+1}-u_j\|_{L^{M_{\frac{n}{\ell}}}(\rn)}\leq \frac{1}{2}\|u_{j}-u_{j-1}\|_{L^{M_{\frac{n}{\ell}}}(\rn)}
\end{align}
for $j \in \N$, whence
\begin{equation}\label{E:decay-mn2}
\|u_{j+1}-u_j\|_{L^{M_{\frac{n}{\ell}}}(\rn)}\leq 2^{-j}.
\end{equation}
Now, observe that  there exists a constant $c_3$ such that
\begin{align}\label{2024-100}
\|u_{j+1}-u_j\|_{W^{1,M_{\frac{n}{\ell}}}(\rn)}
&\leq c_3\|\mathcal N(u_j)-\mathcal N(u_{j-1})\|_{W^{1,M}(\rn)}\\
\nonumber
&\leq c_3\|\mathcal N(u_j)-\mathcal N(u_{j-1})\|_{L^{M}(\rn)}+c_3\|\nabla(\mathcal N(u_j)-\mathcal N(u_{j-1}))\|_{L^{M}(\rn)}.
\end{align}
Equations ~\eqref{E:mn2} and ~\eqref{E:decay-mn2} imply that
\begin{align}
    \label{2024-101}
    \|\mathcal N(u_j)-\mathcal N(u_{j-1})\|_{L^{M}(\rn)} \leq 2^{-j},
\end{align}
provided that $\varepsilon$ is sufficiently small.
On the other hand,
\begin{align}\label{E:lm2}
&\|\nabla(\mathcal N(u_j)-\mathcal N(u_{j-1}))\|_{L^{M}(\rn)}\\
\nonumber
&\leq \|\mathcal N'(u_j)\nabla(u_j-u_{j-1}) \|_{L^M(\rn)} +\|\nabla u_{j-1} (\mathcal N'(u_j)-\mathcal N'(u_{j-1}))\|_{L^M(\rn)}.
\end{align}
The same steps as in \eqref{2024-102} enable one to deduce that
\begin{align}
    \label{2024-103}
\|\mathcal N'(u_j)\nabla(u_j-u_{j-1}) \|_{L^M(\rn)}
\leq \frac{1}{2 c_3}\|u_j-u_{j-1}\|_{W^{1,M_{\frac{n}{\ell}}}(\rn)}
\end{align}
for sufficiently small $\varepsilon$, where $c_3$ denotes the constant from~\eqref{2024-100}.
%
\\
It remains to estimate the second addend on the right-hand side of~\eqref{E:lm2}. 
Thanks to ~\eqref{2024-92}, applied with $v=u_j-u_{j-1}$, and~\eqref{E:decay-mn2} one has that
\begin{align}
    \label{2024-105}
\|\phi(\theta|u_j-u_{j-1}|)\|_{L^{Q_1}(\rn)} \leq   c 2^{-j\gamma}
\end{align}
 for some constant $c$.
From \eqref{eq:N3}, \eqref{2024-92} and \eqref{2024-105} we infer, analogously to \eqref{2024-85},
that
\begin{align}\label{2024-106}
\|\nabla u_{j-1} (\mathcal N'(u_j)-\mathcal N'(u_{j-1}))\|_{L^M(\rn)} \leq c\varepsilon 2^{-j \gamma}.
\end{align}
Combining equations \eqref{2024-100}, \eqref{2024-101}, \eqref{2024-103} and \eqref{2024-106} ensures that 
$$
\|u_{j+1}-u_j\|_{W^{1,M_{\frac{n}{\ell}}}(\rn)}
\leq \frac{1}{2}\|u_{j}-u_{j-1}\|_{W^{1,M_{\frac{n}{\ell}}}(\rn)}+c2^{-j\gamma}
$$
for some constant $c$.
Iterating this inequality we obtain, as in
\eqref{2024-107}:
\begin{align}
    \label{2024-108}
    \|u_{j+1}-u_j\|_{W^{1,M_{\frac{n}{\ell}}}(\rn)}
\leq  c(2^{-j}+ 2^{-j\gamma})
\end{align}
for some constant $c$ and for $j \in \N$.
%
%
Since $\gamma>0$, it follows that $\{u_j\}$ is a Cauchy sequence in $W^{1,M_{\frac{n}{\ell}}}(\rn)$, whose limit $u \in W^{1,M_{\frac{n}{\ell}}}(\rn)$. Moreover, from \eqref{2024-101}, \eqref{E:lm2}, \eqref{2024-103}, \eqref{2024-106}, and \eqref{2024-108} one also deduces that
\begin{align}
    \label{2024-109}
    \|\mathcal N(u_{j+1})-\mathcal N(u_j)\|_{W^{1,M}(\rn)}
\leq  c(2^{-j}+ 2^{-j\gamma}), 
\end{align}
 for some constant $c$ and for $j \in \N$.
Hence,
 $\{\mathcal N(u_j)\}$ is a Cauchy sequence in $W^{1,M}(\rn)$, whose limit $\mathcal N(u)$
 fulfills \eqref{2024-270}.
\end{proof}

\bigskip{}{}

 \par\noindent {\bf Data availability statement.} Data sharing not applicable to this article as no datasets were generated or analysed during the current study.

\section*{Compliance with Ethical Standards}\label{conflicts}

\smallskip
\par\noindent
{\bf Funding}. This research was partly funded by:
\\ (i) GNAMPA   of the Italian INdAM - National Institute of High Mathematics (grant number not available)  (A. Cianchi);
\\ (ii) Research Project   of the Italian Ministry of Education, University and
Research (MIUR) Prin 2017 ``Direct and inverse problems for partial differential equations: theoretical aspects and applications'',
grant number 201758MTR2 (A. Cianchi);
\\ (iii) Research Project   of the Italian Ministry of Education, University and
Research (MIUR) Prin 2022 ``Partial differential equations and related geometric-functional inequalities'',
grant number 20229M52AS, cofunded by PNRR (A. Cianchi);
\\  (iv) Grant no.\ 23-04720S of the Czech Science
Foundation  (L. Slav\'ikov\'a);
\\ (v) Primus research programme PRIMUS/21/SCI/002 of Charles University (L. Slav\'ikov\'a);
\\ (vi) Charles University Research Centre program No.\ UNCE/24/SCI/005 (L. Slav\'ikov\'a).

\bigskip
\par\noindent
{\bf Conflict of Interest}. The authors declare that they have no conflict of interest.

\end{document}